\documentclass[a4paper,11pt,oneside]{amsart}
\usepackage[english]{babel}
\usepackage{amssymb}
\usepackage{amsmath}
\usepackage{amsthm}
\usepackage{geometry}
\usepackage{enumitem}
\usepackage{caption}
\usepackage{url}
\usepackage{tikz}
\usepackage{hyperref}
\hypersetup{hidelinks}

\theoremstyle{plain}
\newtheorem{thm}{Theorem}[section]
\newtheorem*{thm2}{Theorem}
\newtheorem{lem}[thm]{Lemma}
\newtheorem{prop}[thm]{Proposition}
\newtheorem{cor}[thm]{Corollary}

\theoremstyle{definition}
\newtheorem{defn}[thm]{Definition}
\newtheorem{rem}[thm]{Remark}
\newtheorem{remark}[thm]{Remark}

\numberwithin{equation}{section}

\renewcommand{\a}{\alpha}
\renewcommand{\b}{\beta}
\renewcommand{\d}{\delta}
\newcommand{\e}{\varepsilon}
\newcommand{\Om}{\Omega}
\newcommand{\D}{\Phi}
\newcommand{\Dp}{\Phi^+}
\newcommand{\g}{\mathfrak g}
\newcommand{\bC}{\mathbb C}
\newcommand{\bN}{\mathbb N}
\newcommand{\bZ}{\mathbb Z}
\newcommand{\Qq}{\mathbb{Q}(q)}
\renewcommand{\O}{\mathcal{O}}
\newcommand{\exd}{\mathrm{d}}
\newcommand{\id}{\mathrm{id}}
\newcommand{\proj}{\mathrm{proj}}
\newcommand{\supp}{\mathrm{supp}}
\newcommand{\cone}{\mathrm{cone}}
\newcommand{\height}{\mathrm{ht}}
\newcommand{\inv}{^{-1}}
\newcommand{\mto}{\mapsto}
\newcommand{\sseq}{\subseteq}
\newcommand{\Uplus}{U_q(\mathfrak{n}^+)}
\newcommand{\Uminus}{U_q(\mathfrak{n}^-)}
\newcommand{\pair}[2]{\langle #1 , #2 \rangle}
\newcommand{\bfa}{\mathbf{a}}
\newcommand{\bfb}{\mathbf{b}}
\newcommand{\bfc}{\mathbf{c}}
\newcommand{\bfd}{\mathbf{d}}
\newcommand{\bfu}{\mathbf{u}}
\newcommand{\bfv}{\mathbf{v}}
\newcommand{\bfe}{\mathbf{e}}
\newcommand{\bfze}{\mathbf{0}}

\author[A. C.]{Alessandro Carotenuto}
\address[Alessandro Carotenuto]{Department of Mathematics, Faculty of Science, University of Zagreb\\
Bijeni\v{c}ka cesta 30\\ 10000 Zagreb\\ Croatia}
\email{acaroten91@gmail.com}

\author[C. H.]{Christophe Hohlweg}
\address[Christophe Hohlweg]{LaCIM et D\'epartement de Math\'ematiques\\ Universit\'e du Qu\'ebec \`a Montr\'eal\\ CP 8888 Succ. Centre-Ville\\
Montr\'eal, Qu\'ebec, H3C 3P8\\ Canada}
\email{hohlweg.christophe@uqam.ca}

\author[P. P.]{Paolo Papi}
\address[Paolo Papi]{Dipartimento di Matematica\\ Sapienza Universit\`a di Roma\\ P.le Aldo Moro 2, 00185 Roma\\ Italy}
\email{papi@mat.uniroma1.it}

\title{Convex orders in root systems and quantum tangent spaces}

\begin{document}

\begin{abstract}
We classify the reduced decompositions of the longest element $w_0$ of the Weyl group of a semisimple Lie algebra $\mathfrak{g}$ for which the corresponding Lusztig quantum root vectors span a quantum tangent space in $U_q(\mathfrak{g})$. We introduce a combinatorial condition on the convex order induced by a reduced decomposition, tangency, and prove that a reduced decomposition gives rise to a quantum tangent space if and only if it is tangent. The proof relies on a description of the coproduct of quantum root vectors in terms of the convex order, and on a projection in the right weak order attached to each positive root. Up to commutation, tangent reduced decompositions are obtained by concatenating the longest minimal coset representatives along a chain of cominuscule parabolic quotients; in particular there are none in types $E_8$, $F_4$ and $G_2$.
\end{abstract}
\maketitle

\section{Introduction}
The choice of a differential calculus for a noncommutative algebra $A$, viewed as the algebra of some class of functions on a noncommutative manifold,
is a foundational problem in the relatively young field of
noncommutative differential geometry, and a key feature of the theory. This constitutes a fundamental difference from classical differential geometry, where a manifold comes equipped with its own de Rham complex.
When a first order differential calculus $\Omega^1$ on $A$ is given, it always admits a unique maximal prolongation to a differential calculus $\Omega^\bullet$ over $A$. In most of the concrete examples, this reduces the problem to the one of finding a good choice for $\Omega^1$, as the right choice of a differential calculus turns out to be its maximal prolongation. There are however some notable exceptions to this rule, such as Woronowicz's bicovariant calculi for quantum groups \cite{WoroDC}.
In the setting of Drinfel'd--Jimbo quantum groups, and particularly for the quantum flag manifolds $\mathcal{O}_q(G/L_S)$ introduced by Soibelman \cite{SoibelmanQFM}, requiring covariance of the calculus allows for its local description through either the fundamental theorem of Hopf modules or its generalization to quantum homogeneous spaces provided by Takeuchi's equivalence. An influential result in this line of research was established by Heckenberger and Kolb, who gave a $q$-deformation of the classical de Rham complex of the flag manifolds of Hermitian symmetric type \cite{HKdR}, now known as the Heckenberger--Kolb calculi.
These are generalizations of Majid's left-covariant calculus for the Podle\'s sphere $\mathcal{O}_q(S^2)$ \cite{MajidPodles}, and are the essentially unique differential calculi which exhibit at the same time left covariance with respect to the coaction of $\mathcal{O}_q(G)$ and classical dimension of their local part. For these reasons, there is a widespread consensus that these objects are the right quantization of the classical de Rham complex of irreducible flag manifolds.
This consensus has in turn been reinforced by the many results on their noncommutative differential geometry obtained since then. These range from their K\"ahler geometry \cite{MarcoConj,MMF3}, to the theory of connections for their homogeneous vector bundles \cite{HVBQFM}, to their Borel--Weil theory \cite{BwGrass,CDOBBW}. We refer the reader interested in the noncommutative differential geometry of quantum homogeneous spaces to the comprehensive monograph \cite{BeggsMajid:Leabh}.

With most of the theory for quantum flag manifolds of irreducible type already well established, mathematicians have recently started to address the noncommutative differential geometry of quantum flag manifolds beyond the irreducible case, unveiling new interesting phenomena in quantum geometry. A meaningful contribution in this direction was recently given by \'O~Buachalla and Somberg in \cite{ROBPSLusz}, where they exhibit a quantization of the holomorphic calculus of the full flag manifolds of type $A_n$. A follow-up was presented in \cite{COBR25}, where the authors gave a quantization of the full de Rham complex of the full flag manifold of $SU_3$.

The main tools used in \cite{ROBPSLusz} are the quantum root vectors of the quantized enveloping algebra $U_q(\mathfrak{g})$ defined by Lusztig \cite{Lusztig88,Lusztig90_2,LusztisLeabh}. The authors exhibited two choices of a reduced decomposition of the longest element $w_0$ of the Weyl group of $\mathfrak{sl}_{n+1}$ for which the associated sets of quantum root vectors span quantum tangent spaces.
They then showed that the associated first order differential calculi admit maximal prolongations of classical dimension. Following the same approach, the first author, \'O~Buachalla and Pagani gave a quantization of the holomorphic calculus of the full flag manifold of type $B_2$. In this case only one of the two reduced decompositions of $w_0$ gives rise to a quantum tangent space \cite{COBP}.

It is important to note that not every choice of a reduced decomposition of $w_0$ gives a basis of a quantum tangent space. This was the original motivation for the present paper, whose main goal is to understand the combinatorial properties of the reduced decompositions of $w_0$ that give rise to quantum tangent spaces. At the level of the quantum coordinate algebra $\mathcal{O}_q(G)$, the notion of a quantum tangent space is very close to that of a right coideal in $U_q(\mathfrak{g})$. For this reason,
one has to understand how the coproduct of the quantum root vectors of $U_q(\mathfrak{g})$ depends on the choice of a reduced decomposition. This problem is far from trivial, since Lusztig's root vectors are defined by means of an action of the braid group $\mathcal{B}$ of $W$ on $U_q(\mathfrak{g})$ by algebra automorphisms which are not coalgebra maps.
The number of reduced decompositions of $w_0$ grows quickly with the rank of $\mathfrak{g}$; so does the number of their commutation classes, which is what matters here, and whose growth is understood only up to bounds even in the relatively simple case of $A_n$ \cite{FV11}.
For this reason, a direct approach---considering all the reduced decompositions of $w_0$ and computing the coproducts of the corresponding quantum root vectors---is computationally prohibitive, and a more conceptual framework is needed.

A number of authors have explicitly computed the coproducts of quantum root vectors. In \cite{Karchenko14} the author computes the coproducts of quantum root vectors for $U_q(\mathfrak{sp}_{2n})$ and $U_q(\mathfrak{so}_{2n})$, while in \cite{KV18} a coproduct formula for the quantum root vectors in type $G_2$ is given. In \cite{Krahmer03} Kr\"{a}hmer expressed the coproduct of quantum root vectors in terms of $l$-functionals.
In all these works, the quantum root vectors correspond to specific choices of the reduced decomposition of $w_0$.

In the present paper the notion of a convex order on the set of positive roots allows us to address the existence of quantum tangent spaces, as well as their classification, in combinatorial terms.
As shown in Section \ref{sec:coproduct}, the interplay between the convex orders on the set $\Phi^+$ of positive roots of $\mathfrak{g}$ and the coproduct of $U_q(\mathfrak{g})$ lies in the fact that, for the negative quantum root vector attached to $\beta\in \Phi^+$, the coproduct $\Delta: U_{q}(\mathfrak{g})\to U_{q}(\mathfrak{g}) \otimes U_{q}(\mathfrak{g})$ separates the components corresponding to the roots that precede $\beta$ in the convex order from the components corresponding to the roots that follow it, by placing them on the two sides of the tensor product. Furthermore the right weak order on $W$, through the projections introduced below, plays a central role in deciding whether the terms of the coproduct whose left tensor factor has PBW degree at least two vanish, which is what the existence of a quantum tangent space amounts to.

For a reduced decomposition $w_0=s_{i_1}\cdots s_{i_N}$ and $1\leq k\leq N$, let $w_k=s_{i_1}\cdots s_{i_k}$ and $\beta_k=s_{i_1}\cdots s_{i_{k-1}}(\alpha_{i_k})$. Denote by $\Phi(v)=\{\a\in\Dp\mid v^{-1}(\a)\in \D^-\}$ the inversion set of $v$. In Section \ref{sec:projection} we show that, with respect to the right weak order on $W$, the set
$\{v\leq_Rw_k \mid\beta_k\in \Phi(v)\}$
has a unique minimum $\overline{w}_k$, the projection $p(w_k,\beta_k)$ of Section \ref{sec:projection}; it is a join-irreducible element of the weak order, whose unique right descent root is $\beta_k$. The existence of this minimum is the first main step, and it is not automatic: the right weak order is a lattice, but a set of this kind need not have a minimum for an arbitrary root $\beta$. It does when $\beta$ is a right descent root of the element, that is $\beta=-w(\alpha_s)$ with $s$ simple and $\ell(ws)<\ell(w)$, as is the case for $\beta_k=-w_k(\alpha_{i_k})$. We prove this by means of Reading's theory of shards \cite{Re}: the reflecting hyperplanes of $W$ are cut into pieces, the shards, which are in bijection with the join-irreducible elements of the weak order (those with a unique right descent); for a right descent root $\beta$ of $w$ this bijection gives a unique minimum $p(w,\beta)$ of the set above, which is then the interval $[p(w,\beta),w]_R$ (Proposition \ref{p32}; this holds for every finite Coxeter group). Thus every reduced decomposition of $w_0$ comes with canonical elements $\overline{w}_1,\dots,\overline{w}_N$, and the conditions in the main theorem below are conditions on the sets $\Phi(\overline{w}_k)$.

We characterize the reduced decompositions corresponding to quantum tangent spaces in Theorem \ref{thm:combinatorialqts} of Section \ref{sec:qts}, which reads as follows.

\begin{thm2}
Let $w_0=s_{i_1}\cdots s_{i_N}$ be a reduced decomposition and $\beta_1,\dots,\beta_N$ the associated positive roots. The elements
$F^{\bf i}_{\beta_k}K_{\beta_k}$, $1\leq k\leq N$, where the $F^{\bf i}_{\beta_k}$ are the negative quantum root vectors, span a quantum
tangent space if and only if, for every $k$, the minimum $\overline{w}_k$ satisfies the
following two conditions:
\begin{enumerate}
\item[(TR1)] the roots of $\Phi(\overline{w}_k)$, listed in the order in which they occur in
$\beta_1,\dots,\beta_N$, form the sequence of positive roots of a reduced decomposition
of $\overline{w}_k$;
\item[(TR2)] if $\alpha\in \Phi(\overline{w}_k)$ and $\beta_k-\alpha$ is a nonzero sum of
simple roots, then no simple root occurring in $\beta_k-\alpha$ belongs to
$\Phi(\overline{w}_k)$.
\end{enumerate}
\end{thm2}
\noindent This theorem translates a question originating in noncommutative differential geometry into a problem in the combinatorics of root systems.
We use this combinatorial language to characterize, for all types of simple Lie algebras, the reduced decompositions of $w_0$ which satisfy the two properties given above.
\par A reduced decomposition satisfying (TR1) and (TR2) for every $k$ is called
tangent; this notion depends only on the commutation class of the decomposition, as do
the quantum root vectors themselves. Sufficiency holds in every type and is proved
through a stronger condition (TR2$'$), which excludes all the decompositions
$\beta_k=\sum_{p<k}a_p\beta_p+\sum_{q>k}b_q\beta_q$ with $\sum_p a_p\geq2$, the only
ones that could contribute to the relevant part of the coproduct (Propositions
\ref{prop:nobaddec} and \ref{prop:qtscriterion}). Every tangent reduced decomposition is
commutation equivalent to one satisfying (TR2$'$) (Corollary \ref{conj:C1}), and
conversely (TR2$'$) implies tangency (Theorem \ref{thm:tr2ptangent}).
\par Necessity is the harder direction. A non-tangent reduced decomposition always admits
decompositions of $\beta_k$ of the above kind, but one needs one of them to contribute a nonzero
coefficient. A decomposition
$\beta_k=\gamma_1+\gamma_2+\psi$ into positive roots with
$\gamma_1\leq\gamma_2<\beta_k<\psi$ and $\beta_k-\gamma_2$ a root will be called a
\emph{special triple} at $\beta_k$. For a special triple the
coefficient is a product of two structure constants of
$\mathfrak{g}$ and hence nonzero (Proposition \ref{prop:nonvanishing}), and we show that
a special triple occurs, in the given order itself, at the
first root at which tangency fails (Theorem \ref{thm:C2ABCDE}). The main
combinatorial input is a property of tangent roots valid in all types: if (TR1) and
(TR2) hold at $\beta_k$ and at all the roots preceding it, then $\Phi(\overline{w}_k)$ is abelian, and all its elements lie below $\beta_k$ in the root
poset (Proposition \ref{prop:sumfree}); more precisely, $\Phi(\overline{w}_k)$ consists of
the positive roots lying below $\beta_k$ whose coefficient on a suitable simple root equals
one (Lemma \ref{lem:principal}). The proof that a special triple occurs is uniform, and the type of the root system enters only through
the length of root strings: it covers every type except $G_2$ (Theorem
\ref{thm:C2ABCDE}), while in type $G_2$ special triples are exhibited directly (Proposition \ref{prop:FG}).
\par Tangent reduced decompositions are classified in Theorem \ref{t43}: up to
commutation they are exactly the canonical ones, obtained by concatenating the longest
minimal coset representatives along a chain of standard parabolic subgroups
$W=W_0\supset W_1\supset\dots\supset W_n=\{e\}$, each obtained from the previous one by
removing a simple reflection which is cominuscule in it. In particular a tangent
decomposition begins with a simple reflection $s$ such that $\alpha_s$ is cominuscule; in type $B_2$ only the
decomposition beginning with the long simple reflection is tangent, in agreement with the
example recalled above, and in types $E_8$, $F_4$ and $G_2$, which have no cominuscule
simple roots, there are no tangent decompositions at all.
\subsection{Summary of the paper}
The paper consists of two parts. The first one, Sections
\ref{sec:coxeter}--\ref{sec:reflorders},  concerns combinatorial properties  of finite Weyl groups
of all types. Section \ref{sec:coxeter} recalls reduced words, commutation classes,
inversion sets and the weak order, and Section \ref{sec:cominuscule} the root poset and
the longest coset representatives of maximal standard parabolic subgroups, with emphasis
on the cominuscule case. In Section \ref{sec:projection} we introduce the projection of
an element of $W$ along one of its right descent roots and show, by means of Reading's
theory of shards, that it is the unique minimum of the corresponding set and a
join-irreducible element of the weak order. Section \ref{sec:reflorders} contains the
theory of tangent convex orders: the conditions (TR1) and (TR2); the condition
(TR2$'$), which implies tangency (Theorem \ref{thm:tr2ptangent}); the classification of
tangent convex orders up to commutation (Theorem \ref{t43}); the fact that every tangent
convex order is commutation equivalent to one satisfying (TR2$'$) (Corollary
\ref{conj:C1}); and the behavior of projections at tangent roots (Proposition
\ref{prop:sumfree}).
\par The second part concerns quantum groups. Section \ref{sec:fodc} recalls covariant
first order differential calculi, quantum tangent spaces and Drinfel'd--Jimbo quantum
groups, and Section \ref{sec:qrv} Lusztig's quantum root vectors, which depend on the
reduced decomposition only through its commutation class. In Section \ref{sec:coproduct}
we describe the coproduct of quantum root vectors in terms of the convex order induced by
the reduced decomposition, and give a criterion for the non-vanishing of its
coefficients. Section \ref{sec:qts} translates the existence of a quantum tangent space
into the combinatorics of the first part, proves that tangency is sufficient in all
types, and reduces necessity to Theorem \ref{conj:C2}, a statement on root systems.
Finally Section \ref{sec:towards} proves Theorem \ref{conj:C2}, by a uniform argument in
all types other than $G_2$, and then in type $G_2$.
\vskip10pt
\centerline{\large \bf The first half of the title: convex orders in root systems}
\vskip10pt
\section{Weyl groups as Coxeter groups}\label{sec:coxeter}
Let $W$ be a finite Weyl group with set of simple reflections $S$, so that $(W,S)$ is a Coxeter system. Consider a root system $\D$ for $W$ in the reflection representation
$W\to GL(V)$, where $V$ is a Euclidean vector space with scalar product $(\cdot,\cdot)$.
For $\a\in\D$ we denote by $\a^\vee=2\a/(\a,\a)$ the corresponding coroot and we write
\begin{align*}
\langle\b,\a^\vee\rangle=\frac{2(\b,\a)}{(\a,\a)},\qquad \b\in V,
\end{align*}
so that $s_\a(\b)=\b-\langle\b,\a^\vee\rangle\a$; as $\D$ is the root system of a Weyl
group, $\langle\b,\a^\vee\rangle\in\mathbb Z$ for $\b\in\D$. Note that
$\langle\cdot,\cdot^\vee\rangle$ is linear in its first argument only. Let
$\Pi=\{\a_s\mid s\in S\}$ be a simple system. The set of positive roots is $\Dp=\cone_{\D}(\Pi)$, where
$\cone_{\D}(X)=\cone(X)\cap \D,$ and $\cone(X)$ is the cone over $X\subset V$. When $\D$ is irreducible, we denote by $\theta$ its highest root. We denote by $Q=\bZ\Pi$ the root lattice and by $Q^+=\bN\Pi$ the set of nonnegative integer combinations of simple roots (here $\bN=\{0,1,2,\dots\}$), so that $\Dp\subset Q^+$.
\par
In the rest of this section we survey some well-known properties of Coxeter groups; see for instance \cite{BB, Bo, GP, Hu} for general references.
\subsection{Reduced words, length and descents} Denote by $S^*$ the free monoid over the alphabet $S$. The product is given by concatenation. Any element $w\in W$ is a product
$w=s_1s_2\cdots s_k$ for some $s_i\in S$; the word $(s_1,s_2,\ldots ,s_k)\in S^*$ is a {\it reduced word for $w$} if $k$ is minimal for that property. In this case, we also say, by abuse of language, that
$w=s_1s_2\cdots s_k$ is a reduced word. The {\it length function} on $(W,S)$ is
\begin{align*}
\ell:\,\,\ &W\longrightarrow \mathbb N\\
&e\,\,\mapsto\,\, 0\\
&w\mapsto \min\{k\in \mathbb N^*\mid w=s_1s_2\cdots s_k,\,s_i\in S\}
\end{align*}
In short, if $w\in W$, then $\ell(w)$ is the length of any reduced word for $w$.\par
It is well-known that for any $w\in W$ and $s\in S$,
\[\ell(ws)=\ell(w)\pm 1\text{ and } \ell(sw)=\ell(w)\pm 1.\]
For $w\in W$, the {\it right descent set} of $w$ is $D_R(w)=\{s\in S\mid \ell(ws)<\ell(w)\}$, which is the set of simple reflections appearing as the last letter of some reduced word for $w$.
Similarly, the {\it left  descent set} of $w$ is $D_L(w)=\{s\in S\mid \ell(sw)<\ell(w)\}$, which consists of the initial letters of the reduced words for $w$.
\subsection{Commutation relations and fully commutative elements}\label{CRFC} It is well known, under the name of Matsumoto's theorem, that any two reduced words for the same element can be obtained from each other by a sequence of braid moves. The simplest move is the commutation, which we now discuss. Let $\sim$ be the reflexive and transitive closure of the following relation on $S^*$: $(s_1,s_2,\ldots,s_i,s_{i+1},\ldots,s_k)\sim (s_1,s_2,\ldots,s_{i+1},s_{i},\ldots,s_k)$ if $s_is_{i+1}=s_{i+1}s_i$ in $W$. We say that two reduced words
$s_1s_2\cdots s_k, r_1r_2\cdots r_k$ for $w$ are {\it equivalent up to commutation} if $(s_1,s_2,\ldots,s_k)\sim (r_1,r_2,\ldots,r_k)$ in $S^*$. \par An element $w\in W$ is
{\it fully commutative} if its commutation class is a singleton, i.e. any two reduced words for $w$ are equivalent up to commutation.

{\bf Example.} Let $W$ be of type $A_3$.
\begin{itemize}
\item $w=s_2s_1s_3s_2=s_2s_3s_1s_2$ is fully commutative;
\item $w=s_2s_1s_2s_3=s_1s_2s_1s_3=s_1s_2s_3s_1$ is not fully commutative.
\end{itemize}
\subsection{Inversion sets and descent-inversions}
Recall that the  inversion set of $w\in W$ is defined to be
\[\Phi(w)=\Dp\cap w(\D^-)=\{\a\in \Dp\mid w^{-1}(\a)\in \D^-\}.\]
It is well-known that if $s_1s_2\cdots s_k$ is a reduced expression for $w$, then \[\Phi(w)=\{\a_{s_1}, s_1(\a_{s_2}),\ldots, s_1\cdots s_{k-1}(\alpha_{s_k})\};\] in particular,
$|\Phi(w)|=\ell(w)$. The left and right descent sets correspond in $\Phi(w)$ to the following sets
\begin{align*}
\Phi^L(w)&=\Phi(w)\cap \Pi =\{\a_s\mid s\in D_L(w)\}\\
\Phi^R(w)&=\{-w(\a_s)\mid s\in D_R(w)\}\\&=\{\beta\in \Phi(w)\mid \exists s\in S, \Phi(w)\setminus \{\beta\}=\Phi(ws)\}.\end{align*}
The roots in $\Phi^R(w)$ are called {\it descent-inversions}.
\subsection{The weak order}
Let $u,w\in W$. We write $u\leq_Rw$ if there is a reduced word $(s_1,\ldots,s_k)$ for $w$ such that $(s_1,\ldots,s_i)$ is a reduced word for $u$ for some $i$, $0\leq i\leq k$. We say in such a case that {\it $u$ is a prefix of $w$}. The relation $\leq_R$ is a partial order on $W$ called the {\it (right) weak order} on $W$. Similarly, we define the
{\it left weak order} by $v\leq _Lw$ if there exists a reduced word for $v$ which is a suffix of a reduced word for $w$. We have
\[v\leq_L w \iff v^{-1}\leq _R w^{-1}.\]
Bj\"orner showed that the poset $(W,\leq_R)$ is a complete lattice when $W$  is finite: see \cite[Chapter 3]{BB}. In other words, any non-empty subset $X\subset W$ admits:
\begin{itemize}
\item a meet $\bigwedge_R X=\sup\{g\in W\mid g\leq_R x,\,\forall\,x\in X\}$;
\item a join $\bigvee_R X=\inf\{g\in W\mid g\geq_R x,\,\forall\,x\in X\}$.
\end{itemize}
The {\it longest element} of $W$ is $w_0=\bigvee_R S$. \par
The following well-known proposition summarizes useful properties of the weak order and inversion sets; see e.g. \cite{HL} for more details. We say that $w=uv$ is a {\it reduced product} if $\ell(w)=\ell(u)+\ell(v)$.
\begin{prop}\label{p11} Let $u,v,w\in W$.
\begin{enumerate}
\item If  $w=uv$ is a  reduced product, then $\Phi(w)=\Phi(u)\sqcup u(\Phi(v))$.
\item $u\leq_R w\iff \Phi(u)\subseteq \Phi(w)$.
\end{enumerate}
Moreover a subset $A\subseteq \Dp$ is an inversion set, i.e. $A=\Phi(w)$ for some $w\in W$, if and only if $A$ is biconvex, i.e. both $A$ and $\Dp\setminus A$ are convex\footnote{$B$ is convex if
$\cone_\D(X)\subseteq B$ for all $X\subseteq B$.}. In particular, the complement $\Dp\setminus \Phi(w)=\Phi(ww_0)$ of an inversion set is again an inversion set.
\end{prop}
\subsection{Standard parabolic subgroups and longest coset representatives}
Let $I\subseteq S$. The subgroup $W_I$ of $W$ generated by $I$ is called a {\it standard parabolic subgroup}. It is well-known that $(W_I,I)$ is also a Coxeter system with length function $\ell_{|W_I}$. The left and right weak order restrict to $W_I$:
\begin{itemize}
\item the longest element of $W_I$ is $w_{0,I}=\bigvee_R I=\bigvee_L I$.
\item $W_I$ is an interval for both the left and the right weak order:
\[W_I=[e,w_{0,I}]_L=[e,w_{0,I}]_R.\]
\end{itemize}
Moreover, $\Pi_I=\{\a_s\mid s\in I\}$ is a simple system for $(W_I,I)$ with root system $\D_I=\D\cap \mathbb R \Pi_I$ and positive root system
$\D_I^+=\Dp\cap \D_I$. The inversion set of $w\in W_I$ satisfies $\Phi(w)\subset \Dp_I$, and any reduced word for $w$ has all letters in $I$.\par
For $I\subset S$, the set of minimal left coset representatives of $W/W_I$ is
\[X_I=\{x\in W\mid \ell(xs)>\ell(x)\,\forall\,s\in I\}.\]
In short, any $w\in W$ has a unique decomposition into a reduced product $w=uv,\,u\in X_I,\,v\in W_I$. We recall here two useful statements; see for instance \cite{GP}, where they are stated for right cosets.
\begin{lem}[Deodhar's Lemma]\label{l12} Let $I\subset S$. If $x\in X_I$ and $s\in S$, then one of the following statements holds:
\begin{enumerate}
\item $s\in D_L(x)$ and $sx\in X_I$;
\item  $s\notin D_L(x)$ and $sx\in X_I$;
\item $s\notin D_L(x)$ and $sx=xr$ with $r\in I$.
\end{enumerate}
\end{lem}
Let $I\subset S$. Introduce the element \[d_I=w_0w_{0,I},\]
called the longest coset representative of $W/W_I$. Its main properties are displayed in the following proposition.
\begin{prop} \label{p13}\
\begin{enumerate}
\item $X_I=[e,d_I]_L$ is an interval in the left weak order, i.e. any element of $X_I$ is a suffix of $d_I$.
\item ${d_I}^{-1}=d_{w_0Iw_0}=w_0d_Iw_0.$
\item  $d_IW_Id_I^{-1}=W_{w_0Iw_0}.$
\item $\Phi(d_I)=\Dp\setminus \Dp_{w_0Iw_0}=\Dp\setminus w_0(\D^-_I)$.
\item $D_R(d_I)=S\setminus I$ and  $D_L(d_I)=S\setminus w_0Iw_0$.
\end{enumerate}
\end{prop}
The two following propositions are well-known (see for instance \cite[Chapter 2]{GP}).
\begin{prop}\label{p14} Let $I\subseteq J\subseteq S$. Then
\begin{itemize}
\item $X_I^J=X_I\cap W_J$ is the set of minimal coset representatives of $W_J/W_I$.
\item $X_I^J=[e,d^J_I]_L,$ where $d_I^J=w_{0,J}w_{0,I}.$
\end{itemize}
\end{prop}
\begin{prop}\label{p15} Let $w\in W$ and $I=D_R(w)$. Then $w=xw_{0,I}$ with $x\in X_I$. Moreover  if $J\subset S$ is such that $w=yw_{0,J}$ with $y\in X_J$, then $J\subseteq I$.
\end{prop}
We end with a last proposition.
\begin{prop}\label{p16} Let $w\in W$ and $I\subseteq S$ be such that $d_I\leq_R w$. Then
\begin{enumerate}
\item There exists  $v\in W_I$ such that the product $w=d_Iv$ is reduced.
\item  Let $w=s_1\cdots s_k$ be any reduced word and let $1\leq i\leq k$ be such that $s_i\cdots s_k\in X_I$. Then $\beta_j=s_1\cdots s_{j-1}(\a_{s_j})\in \Phi(d_I)$ for all $i\leq j\leq k$.
\end{enumerate}
\end{prop}
\begin{proof} (1) By definition of $\le_R$ we have $w=d_Iv$ with $\ell(w)=\ell(d_I)+\ell(v)$ for some $v\in W$, and we have to show that $v\in W_I$. By Proposition \ref{p11} and Proposition \ref{p13}(4), $\Phi(ww_0)=\Dp\setminus\Phi(w)\subseteq\Dp\setminus\Phi(d_I)=\Dp_{w_0Iw_0}$. An element $u\in W$ with $\Phi(u)\subseteq \Dp_J$ lies in $W_J$: if $u=t_1\cdots t_m$ is a reduced word, then $\a_{t_1}\in\Phi(u)\cap\Pi\subseteq\Pi_J$, so that $t_1\in J$, and $\Phi(t_1u)=t_1(\Phi(u)\setminus\{\a_{t_1}\})\subseteq\Dp_J$, so that one concludes by induction on $\ell(u)$. Hence $ww_0\in W_{w_0Iw_0}=w_0W_Iw_0$, so that $w\in w_0W_I=d_IW_I$ and $v\in W_I$.
\par (2) Let $i\leq j\leq k$. Then
\[d_I^{-1}(\beta_j)=vw^{-1}(\beta_j)=-vs_k\cdots s_{j+1}(\a_{s_j})\in \D^-,\]
since $s_j\cdots s_kv^{-1}$ is a reduced product. The latter claim follows since  $s_j\cdots s_k\in X_I$ (by Proposition \ref{p13}(1), $X_I$ being closed under taking suffixes) and $v^{-1}\in W_I$.
\end{proof}
\section{Longest coset representatives of maximal standard parabolic subgroups}\label{sec:cominuscule}
{\it In this section, we need $W$ to be irreducible and $\D$ to be a crystallographic root system. In the second part of the paper, $\D$ will arise as the root system of
a simple finite-dimensional complex Lie algebra $\g$ with respect to a fixed Cartan subalgebra $\mathfrak h$.} \par Recall that each positive root $\beta$ can be written as $\beta=\sum_{s\in S} n_s\a_s$ with $n_s\in \mathbb N$. The {\it height} of $\beta$ is $\height(\beta)=\sum_{s\in S} n_s\in \mathbb N^*:=\mathbb N\setminus\{0\}$. The support of $\beta$ is $\supp(\beta)=\{\a_s\mid n_s>0\}$. The coefficient of the simple root $\a_s\in \Pi$ in $\beta$ is denoted by $c_s(\beta)$ (in other words,
if $\beta=\sum_{s\in S} n_s\a_s$, then $c_s(\beta)=n_s$); more generally, $c_s(\mu)$ denotes the coefficient of $\a_s$ in any $\mu\in Q$.
\subsection{The root poset} The root poset $(\Dp,\leq _0)$ is defined as follows:
$\a\leq_0\beta$ if $\beta-\a\in Q^+$. The root poset is ranked by the height. The set of minimal elements is $\Pi$  and the highest root $\theta$ is the unique maximal element.
\begin{prop}\label{p21}
Let $s\in S$ and write $(s)=S\setminus \{s\}$. The inversion set of the longest coset representative $d_{(s)}$ of the maximal standard parabolic subgroup $W_{(s)}=W_{S\setminus\{s\}}$ is
\[\Phi(d_{(s)})=[-w_0(\a_s),\theta]_0,\]
an  interval in the root poset.
\end{prop}
\begin{proof} We know from Proposition \ref{p13} that $\Phi(d_{(s)})=\Dp\setminus \Dp_{(w_0sw_0)}$.  Since $\Pi\cap (\Dp\setminus \Dp_{(w_0sw_0)})=\Pi\setminus \Pi_{(w_0sw_0)}=\{-w_0(\a_s)\}$, we have
that $\beta\in \Dp\setminus \Dp_{(w_0sw_0)}$ if and only if $-w_0(\a_s)=\a_{w_0sw_0}\in \supp(\beta)$. Therefore
\[\Phi(d_{(s)})=\{\beta\in \Dp\mid -w_0(\a_s)\in \supp(\beta)\}=[-w_0(\a_s),\theta]_0.\]
\end{proof}
\subsection{Cominuscule roots and fully commutative elements}
A simple root $\a_s\in \Pi$ is said to be {\it cominuscule} if $c_s(\theta)=1$; cominuscule roots are related to minuscule fundamental coweights: if $\a_s$ is cominuscule then the fundamental coweight $\varpi_s^\vee$ is minuscule (i.e., $W$ acts
 transitively on the weights of the irreducible highest weight $\g^\vee$-module with highest weight $\varpi_s^\vee$) and $W_{(s)}$ is the stabilizer of $\varpi_s^\vee$. We then say that $X_{(s)}$ is minuscule.
 It is readily seen that $\a_s$ is cominuscule if and only if $\a_{w_0sw_0}=-w_0(\a_s)$ is cominuscule.

 Stembridge proved in \cite[Theorem 6.1(b)]{St} the following statement.
 \begin{prop}[Stembridge]\label{p22} If $\a_s\in \Pi$ is cominuscule, then every element $u\in X_{(s)}$ is fully commutative. In particular, $d_{(s)}$ is fully commutative.
 \end{prop}
 {\bf Example.} Let $\D$ be a root system of type $B_3$, realized in $\mathbb R^3$ with the simple system $\Pi=\{\a_1=\varepsilon_1-\varepsilon_2,\ \a_2=\varepsilon_2-\varepsilon_3,\ \a_3=\varepsilon_3\}$ of \cite{Bo}.  Since $\theta=\a_1+2\a_2+2\a_3$, the only cominuscule root is $\a_1$. Observe that, since  $w_0=-\mathrm{Id}$, we have $w_0sw_0=s$ for all $s\in S$, so that by Proposition \ref{p21} the inversion sets $\Phi(d_{(s)})$ are the intervals $[\a_s,\theta]_0$ shown in Figure \ref{fig:B3}.
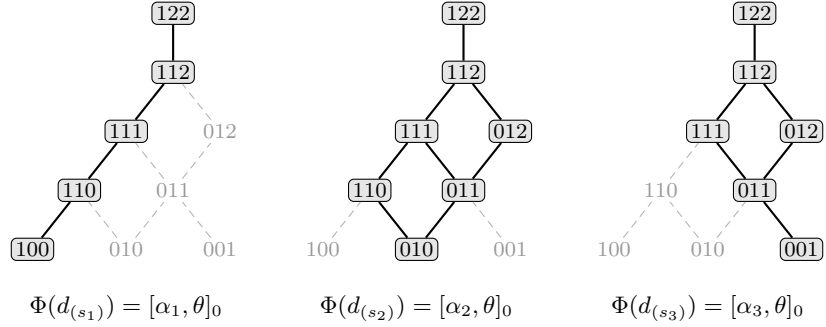
\begin{figure}[htb]
\centering
\begin{tikzpicture}[x=0.62cm,y=0.78cm,
  rin/.style={draw,rounded corners=2pt,fill=black!10,inner sep=1.6pt,font=\scriptsize},
  rout/.style={text=black!40,inner sep=1.6pt,font=\scriptsize},
  ein/.style={thick},eout/.style={black!30,densely dashed}]
\begin{scope}[xshift=0.00cm]
  \node[rout] (n001) at (4,0) {$001$};
  \node[rout] (n010) at (2,0) {$010$};
  \node[rin] (n100) at (0,0) {$100$};
  \node[rout] (n011) at (3,1) {$011$};
  \node[rin] (n110) at (1,1) {$110$};
  \node[rin] (n111) at (2,2) {$111$};
  \node[rout] (n012) at (4,2) {$012$};
  \node[rin] (n112) at (3,3) {$112$};
  \node[rin] (n122) at (3,4) {$122$};
  \draw[ein] (n100)--(n110);
  \draw[eout] (n010)--(n110);
  \draw[eout] (n010)--(n011);
  \draw[eout] (n001)--(n011);
  \draw[ein] (n110)--(n111);
  \draw[eout] (n011)--(n111);
  \draw[eout] (n011)--(n012);
  \draw[ein] (n111)--(n112);
  \draw[eout] (n012)--(n112);
  \draw[ein] (n112)--(n122);
  \node[font=\footnotesize] at (2,-1) {$\Phi(d_{(s_1)})=[\a_1,\theta]_0$};
\end{scope}
\begin{scope}[xshift=3.84cm]
  \node[rout] (n001) at (4,0) {$001$};
  \node[rin] (n010) at (2,0) {$010$};
  \node[rout] (n100) at (0,0) {$100$};
  \node[rin] (n011) at (3,1) {$011$};
  \node[rin] (n110) at (1,1) {$110$};
  \node[rin] (n111) at (2,2) {$111$};
  \node[rin] (n012) at (4,2) {$012$};
  \node[rin] (n112) at (3,3) {$112$};
  \node[rin] (n122) at (3,4) {$122$};
  \draw[eout] (n100)--(n110);
  \draw[ein] (n010)--(n110);
  \draw[ein] (n010)--(n011);
  \draw[eout] (n001)--(n011);
  \draw[ein] (n110)--(n111);
  \draw[ein] (n011)--(n111);
  \draw[ein] (n011)--(n012);
  \draw[ein] (n111)--(n112);
  \draw[ein] (n012)--(n112);
  \draw[ein] (n112)--(n122);
  \node[font=\footnotesize] at (2,-1) {$\Phi(d_{(s_2)})=[\a_2,\theta]_0$};
\end{scope}
\begin{scope}[xshift=7.69cm]
  \node[rin] (n001) at (4,0) {$001$};
  \node[rout] (n010) at (2,0) {$010$};
  \node[rout] (n100) at (0,0) {$100$};
  \node[rin] (n011) at (3,1) {$011$};
  \node[rout] (n110) at (1,1) {$110$};
  \node[rin] (n111) at (2,2) {$111$};
  \node[rin] (n012) at (4,2) {$012$};
  \node[rin] (n112) at (3,3) {$112$};
  \node[rin] (n122) at (3,4) {$122$};
  \draw[eout] (n100)--(n110);
  \draw[eout] (n010)--(n110);
  \draw[eout] (n010)--(n011);
  \draw[ein] (n001)--(n011);
  \draw[eout] (n110)--(n111);
  \draw[ein] (n011)--(n111);
  \draw[ein] (n011)--(n012);
  \draw[ein] (n111)--(n112);
  \draw[ein] (n012)--(n112);
  \draw[ein] (n112)--(n122);
  \node[font=\footnotesize] at (2,-1) {$\Phi(d_{(s_3)})=[\a_3,\theta]_0$};
\end{scope}
\end{tikzpicture}
\caption{The root poset of type $B_3$, with the labeling of the example
($\a_3$ short); the root $c_1\a_1+c_2\a_2+c_3\a_3$ is written $c_1c_2c_3$. In each copy the
framed roots form the interval $[\a_s,\theta]_0=\Phi(d_{(s)})$, whose cardinality is the length
of $d_{(s)}$.}\label{fig:B3}
\end{figure}

We have
\[
d_{(s_1)}=s_1s_2s_3s_2s_1, \qquad
d_{(s_2)}=s_2s_3s_1s_2s_3s_1s_2,\qquad
d_{(s_3)}=s_3s_2s_1s_3s_2s_3.\]
The element $d_{(s_1)}$ has a unique reduced word, hence it is fully commutative, while $d_{(s_2)}$ is not fully commutative, since $d_{(s_2)}=s_2s_3s_1s_2s_3s_1s_2=s_2s_3s_1s_2s_1s_3s_2=s_2s_3s_2s_1s_2s_3s_2$. Note that
the only two reduced words for $d_{(s_3)}$, namely $s_3s_2s_1s_3s_2s_3$ and $s_3s_2s_3s_1s_2s_3$, are in the same commutation class, so that $d_{(s_3)}$ is fully commutative even though $\a_3$ is not cominuscule in $B_3$
(the corresponding simple root is cominuscule in type $C_3$, which has the same Weyl group).
 \section{Projection along right descents}\label{sec:projection}
Let $w\in W$ and $\beta\in \Phi^R(w)$. Recall from Section \ref{sec:coxeter} that $\beta=-w(\a_s)$ for some right descent $s\in D_R(w)$. We consider the following subset of $W$:
\[P(w,\beta)=\{g\in W\mid g\leq_R w,\, \beta\in \Phi(g)\}.\]
Obviously $w\in P(w,\beta)$, hence $P(w,\beta)\ne \emptyset$.
\begin{defn}\label{d51} The {\it projection of $w$ along $\beta$} is the element
\[p(w,\beta)={\bigwedge}_R P(w,\beta)={\bigwedge}_R\{u\in W\mid u\leq_R w,\ \beta \in \Phi(u)\}.\]
\end{defn}
{\bf Example.} Let $\D$ be a root system of type $A_3$. Consider $w=s_2s_1s_2s_3=s_1s_2s_1s_3=s_1s_2s_3s_1$. We have
\[D_R(w)=\{s_1,s_3\}\quad\text{and}\quad \Phi^R(w)=\{s_1s_2(\a_1)=\a_2,\ s_1s_2(\a_3)=\a_1+\a_2+\a_3=\theta\}.\] The interval $[e,w]_R$ is shown in Figure \ref{fig:A3}.
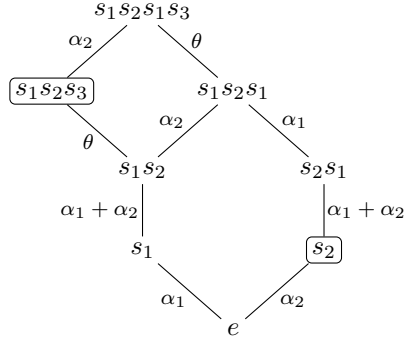
\begin{figure}[htb]
\centering
\begin{tikzpicture}[x=1.2cm,y=1.05cm,
  el/.style={inner sep=2pt,font=\small},
  pmin/.style={draw,rounded corners=2pt,inner sep=2pt,font=\small},
  lab/.style={font=\scriptsize,inner sep=1pt}]
\node[el]   (e)    at (0,0)  {$e$};
\node[el]   (s1)   at (-1,1) {$s_1$};
\node[pmin] (s2)   at (1,1)  {$s_2$};
\node[el]   (s12)  at (-1,2) {$s_1s_2$};
\node[el]   (s21)  at (1,2)  {$s_2s_1$};
\node[pmin] (s123) at (-2,3) {$s_1s_2s_3$};
\node[el]   (s121) at (0,3)  {$s_1s_2s_1$};
\node[el]   (w)    at (-1,4) {$s_1s_2s_1s_3$};
\draw (e)--node[lab,below left]{$\a_1$}(s1);
\draw (e)--node[lab,below right]{$\a_2$}(s2);
\draw (s1)--node[lab,left]{$\a_1+\a_2$}(s12);
\draw (s2)--node[lab,right]{$\a_1+\a_2$}(s21);
\draw (s12)--node[lab,below left]{$\theta$}(s123);
\draw (s12)--node[lab,above left]{$\a_2$}(s121);
\draw (s21)--node[lab,above right]{$\a_1$}(s121);
\draw (s123)--node[lab,above left]{$\a_2$}(w);
\draw (s121)--node[lab,above right]{$\theta$}(w);
\end{tikzpicture}
\caption{The interval $[e,w]_R$ for $w=s_1s_2s_1s_3$ in type $A_3$. Each
covering relation $u\lessdot us$ is labeled by the root $u(\a_s)$, so that $\Phi(u)$ is the
set of labels along any chain from $e$ to $u$, and $P(w,\beta)$ consists of the elements $u$
whose chains from $e$ contain an edge labeled $\beta$. The framed elements are
$p(w,\a_2)=s_2$ and $p(w,\theta)=s_1s_2s_3$.}\label{fig:A3}
\end{figure}
Hence
\begin{align*}
&P(w,\a_2)=\{s_2,s_2s_1,s_1s_2s_1,s_1s_2s_1s_3\}=[s_2,s_1s_2s_1s_3]_R\text{ and } p(w,\a_2)=s_2,\\
&P(w,\theta)=\{s_1s_2s_3,s_1s_2s_1s_3\}=[s_1s_2s_3,s_1s_2s_1s_3]_R\text{ and } p(w,\theta)=s_1s_2s_3.
 \end{align*}
We observe in the above example that $p(w,\beta)\in P(w,\beta)$, which turns out to be true in general as a direct consequence of Reading's theory of shards \cite{Re}.

\begin{prop}\label{p32} Let $w\in W$ and $\beta\in \Phi^R(w)$. Then the projection $p(w,\beta)$ lies in $P(w,\beta)$, and it is join-irreducible with $\Phi^R(p(w,\beta))=\{\beta\}$. Moreover
\begin{itemize}
\item $P(w,\b)=[p(w,\b),w]_R$;
\item if $u\in P(w,\b)$ and $\Phi^R(u)=\{\b\}$, then $u=p(w,\b)$.
\end{itemize}\end{prop}

\begin{proof} The proposition is a consequence of \cite[Proposition 3.3, Lemma 3.5]{Re}, stated in the language of shards. We explain how to translate the language of \cite{Re} into ours. We consider the Coxeter hyperplane arrangement
$\mathcal A_W=\{H_\a\mid \a\in \Dp\}$. The connected components of $V\setminus \bigcup_{\a\in\Dp}H_\a$ are called regions. The unique region $C$ bounded by the hyperplanes $H_\a$, $\a\in \Pi$, is called the fundamental region\footnote{Indeed, its closure is a fundamental domain for the action of $W$ on $V$.}, and the assignment $w\mapsto w(C)$ is a bijection between $W$ and the set of regions of $\mathcal A_W$. Let $R=w(C)$ be a region of
$\mathcal A_W$; the separation set of $R$ is $\mathcal S(R)=\{H_\a\mid H_\a \text{ separates }R\text{ from }C\}$. It is well known that
$\mathcal S(w(C))=\{H_\a\mid \a\in \Phi(w)\}$. The poset of regions is defined by $u(C)\leq w(C)$ if $\mathcal S(u(C))\subseteq \mathcal S(w(C))$, so that
$u(C)\leq w(C)$ if and only if $u\leq_Rw$, and the poset of regions is isomorphic to $(W,\leq_R)$. An element $w\in W$ is join-irreducible (i.e., $w=u\vee_Rv$ implies $u=w$ or $v=w$) if and only if $|D_R(w)|=1$. In terms of regions, this means that $w(C)$ has a unique lower facet, i.e. a unique facet lying in a hyperplane $H_{\b}\in \mathcal S(w(C))$, and then $\Phi^R(w)=\{\b\}$. The shards are obtained by cutting the hyperplanes of $\mathcal A_W$ into pieces, according to the simple hyperplanes of all dihedral reflection subgroups (see \cite[Section 3]{Re}). An upper region of a shard $\Sigma\subseteq H_\b$ is a region $R$ having a facet contained in $\Sigma$ and such that $H_\b\in\mathcal S(R)$. By \cite[Proposition 3.3]{Re}, every shard $\Sigma$ has a unique minimal upper region $J(\Sigma)$, which is join-irreducible, and $\Sigma\mapsto J(\Sigma)$ is a bijection between shards and join-irreducible elements of $(W,\leq_R)$.
\par Now let $\b\in\Phi^R(w)$. Then $H_\b\in\mathcal S(w(C))$ and $w(C)$ has a facet in $H_\b$; this facet lies in a unique shard $\Sigma\subseteq H_\b$, of which $w(C)$ is an upper region. By \cite[Lemma 3.5]{Re}, $J(\Sigma)$ is the unique minimal element of the set of regions $\{A\mid A\leq w(C)\text{ and } H_\b\in\mathcal S(A)\}$. In $(W,\leq_R)$, this says that
$P(w,\b)$ has a unique minimum $p$, namely the element with $p(C)=J(\Sigma)$; thus $p(w,\beta)=\bigwedge_RP(w,\b)=p\in P(w,\b)$. As $J(\Sigma)$ is join-irreducible and has a facet in $H_\b$ with $H_\b\in\mathcal S(J(\Sigma))$, we get $\Phi^R(p)=\{\b\}$. (This last fact also follows directly from minimality: if $\gamma\in\Phi^R(p)$ with $\gamma\neq\b$, the lower cover $ps$ of $p$ with $\Phi(ps)=\Phi(p)\setminus\{\gamma\}$ would lie in $P(w,\b)$.)
\par The first item of the ``moreover'' part follows, since $P(w,\b)$ is an up-set of $[e,w]_R$. For the second one, let $u\in P(w,\b)$ with $\Phi^R(u)=\{\b\}$, so that $u$ has a unique lower cover $us$ in the weak order, with $\Phi(us)=\Phi(u)\setminus\{\b\}$. If $p<_Ru$, then $p\leq_R us$, whence $\b\in \Phi(p)\subseteq \Phi(us)$, a contradiction. Hence $u=p$.
\end{proof}
\begin{rem} The proof does not use the hypothesis that $W$ is a Weyl group: the above proposition holds for every finite Coxeter group. We believe it to be true for infinite Coxeter groups as well.
\end{rem}
\begin{lem}\label{lem:commproj}
Let $w=s_1\cdots s_k s_{k+1}$ be a reduced word with $s_ks_{k+1}=s_{k+1}s_k$, let $\b=s_1\cdots s_{k-1}(\a_{s_k})$, and put $x=s_1\cdots s_k$. Then $\b\in \Phi^R(x)\cap \Phi^R(w)$ and $p(x,\b)=p(w,\b)$.
\end{lem}
\begin{proof}
One has $-x(\a_{s_k})=\b$ and, since $s_{k+1}$ fixes $\a_{s_k}$, also $-w(\a_{s_k})=-xs_{k+1}(\a_{s_k})=\b$; so $\b$ is a descent-inversion of both. As $x\leq_R w$, we have $p(x,\b)\in P(x,\b)\subseteq P(w,\b)$, and $\Phi^R(p(x,\b))=\{\b\}$ by Proposition \ref{p32}. The second item of the ``moreover'' part of Proposition \ref{p32}, applied to $w$, gives $p(x,\b)=p(w,\b)$.
\end{proof}

We end this section by describing how to compute $p(w,\beta)$. Let $w\in W$ and $\b\in \Phi^R(w)$. Let $I=D_R(w)$, so by Proposition \ref{p15} there is $x\in X_I$ such that $w=xw_{0,I}$ is a reduced product.
\begin{itemize}
\item If $|I|=1$, we are done and $p(w,\beta)=w$ by Proposition \ref{p32}.
\item If $|I|\geq 2$, then, since $\b\in \Phi^R(w)$, there is $s\in D_R(w)=I$ such that $\beta=-w(\a_s)$. Let $r=w_{0,I}sw_{0,I}\in I$ and set $w'=xr$. Then $x(\a_r)=x(-w_{0,I}(\a_s))=-xw_{0,I}(\a_s)=-w(\a_s)=\b$, so that $\b=-w'(\a_r)\in\Phi^R(w')$. Therefore $w'\in P(w,\b)$ and $\ell(w')<\ell(w)$; moreover $p(w',\b)=p(w,\b)$, because $p(w,\b)\leq_R w'$ by Proposition \ref{p32}. We replace $w$ by $w'$ and iterate, until we reach an element $w'$ with $|D_R(w')|=1$, which is $p(w,\b)$.\end{itemize}
\par\noindent{\bf Example.} In type $A_3$ consider $w=s_2s_1s_2s_3s_2$. Then $I=D_R(w)=\{s_2,s_3\}$ and $\Phi^R(w)=\{\a_2+\a_3,\a_1\}$. To compute $p(w,\a_2+\a_3)$ observe that
$\a_2+\a_3=s_2s_1s_2s_3(\a_2)$, so that $s=s_2$. We have $w=s_2s_1\cdot s_2s_3s_2\in X_I\cdot w_{0,I}$. Since $-s_2s_3s_2(\a_2)=\a_3$, we have $r=s_3$ and we set $w'=s_2s_1s_3$. We have $D_R(w')=\{s_1,s_3\}$. Repeating the process we obtain $w''= s_2s_3=
p(w,\a_2+\a_3)$.
\section{Reflection orders}\label{sec:reflorders}

\subsection{Convex orders on positive roots}\label{sec:convex}
Recall from \cite{Dy} the following definition. Let $T$ denote the set of reflections of a Coxeter system $(W,S)$.
\begin{defn} A total order $<$ on $T$ is called a {\it reflection order} if, for
any dihedral reflection subgroup $W'$ of $W$ with canonical generators $r,s$, either $r<rsr<\ldots<srs<s$ or $s<srs<\ldots<rsr<r$.
\end{defn}
The following characterization is of crucial importance for us.
\begin{prop}[{\cite[(2.2)]{Dy}}]\label{pp} Identify reflections with positive roots. Then $<$ is a reflection order on $\Dp$ if and only if, whenever $\a,\b\in\Dp$ and $\gamma\in \cone_\D(\a,\b)\setminus\{\a,\b\}$, either $\a<\gamma<\b$ or $\b<\gamma<\a$.
\end{prop}
\begin{rem} Reflection orders have appeared in the literature under various names (normal orders in \cite{AST79}, compatible orders in \cite{Papi}). In applications to Lie theory the most common name is {\it convex orders}, and we shall adhere to this terminology.
Note also that if $\D$ is crystallographic, then the condition displayed in Proposition \ref{pp} simplifies to: if $\a,\b\in\Dp$ and $\a+\b\in\Dp$ then either $\a<\a+\b<\b$ or $\b<\a+\b<\a$ (cf. \cite{Papi}).
\end{rem}

It is known, see e.g. \cite{Dy},  that convex  orders are in bijection with reduced words  for $w_0$:
\begin{itemize}
\item if $w_0=s_1s_2\cdots s_N$ is a reduced word for $w_0$, set
\begin{equation}\label{ordb}\beta_j=s_1s_2\cdots s_{j-1}(\a_{s_j}), \quad 1\leq j \leq N.\end{equation}
Then $\beta_1<\ldots<\b_N$ is a convex order. We will use this notation throughout the paper, and we will say that the order is {\it induced} by the chosen reduced word;
\item if $\beta_1<\ldots<\b_N$ is a convex order, then necessarily $\beta_1=\a_{s_1}\in\Pi$ for some $s_1\in S$. Then $s_1(\beta_2)=\a_{s_2}$ for some $s_2\in S$, and so on; it turns out that $w_0=s_1\cdots s_N$ is a reduced word inducing $<$.
\end{itemize}
Let $<$ be a convex order and $w_0=s_1s_2\cdots s_N$ the corresponding reduced word.
An {\it initial section of $<$} is a subset $A\subseteq \Dp$ such that $\a<\b$ for all $\a\in A$ and $\b\in \Dp\setminus A$. Since $<$ is a total order, an initial section is a down-set, hence an initial segment of $\b_1<\ldots<\b_N$: the initial sections of $<$ are precisely the sets $\Phi(s_1\cdots s_k)=\{\b_1,\ldots,\b_k\}$, $0\leq k\leq N$. In particular they are biconvex and, by \cite[Lemma 2.11]{Dy}, a subset $A\subseteq\Dp$ is an initial section of {\it some} convex order if and only if $A=\Phi(w)$ for some $w\in W$, i.e., if and only if $A$ is biconvex; in that case the restriction to $A$ of any convex order having $A$ as an initial section is induced, as in \eqref{ordb}, by a reduced word $r_1\cdots r_k$ for $w$.
We also say that {\it $w$ is an initial section of $<$} when $\Phi(w)$ is.
\par\noindent{\bf Example.} In type $A_3$, let $<$ be associated to the reduced word $w_0=s_1s_3s_2s_1s_3s_2$, so that $\Dp$ is ordered by
\[\a_1<\a_3<\theta<\a_2+\a_3<\a_1+\a_2<\a_2.
\]
Then $s_1s_3s_2$ is an initial section of $<$, since $\Phi(s_1s_3s_2)=\{\a_1,\a_3,\theta\}$ is ordered by $<$ as $\a_1<\a_3<\theta$, which is the sequence of roots of the reduced word $s_1s_3s_2$.
 But $s_3s_2s_1$ is not an initial section, since $\Phi(s_3s_2s_1)=\{\a_3,\a_2+\a_3,\theta\}$ does not contain $\a_1$, although $\a_1<\a_3$. (Moreover $s_3s_2s_1$ has a unique reduced expression, whose sequence of roots is $\a_3,\ \a_2+\a_3=s_3(\a_2),\ \theta=s_3s_2(\a_1)$, whereas $\theta<\a_2+\a_3$.)

\subsection{Tangent convex orders}\label{tco}
Let $<$ be a convex order and $w_0=s_1\cdots s_N$  be
the corresponding reduced word.
We write $x_0^<=e$ and for $1\leq k\leq N$
\begin{equation}\label{notaz} x_k^<=s_1\cdots s_k,\qquad \b_k^<=s_1\cdots s_{k-1}(\a_{s_{k}}).\end{equation}
Observe that $\b_k^<\in \Phi^R(x_k^<)$. We define the {\it projection of $\b\in \Dp$ according to $<$} by
\begin{equation}\label{proj}
p_<(\beta)=p(x_k^<,\b_k^<)\text{ if $\b=\b_k^< $ with $1\leq k\leq N$.}\end{equation}

\begin{defn} A convex order $<$ on $\Dp$ is {\it tangent} if it satisfies the following two conditions for every $\b\in\Dp$:
\begin{enumerate}
\item[(TR1)] $\Phi(p_<(\b))$, listed in increasing order with respect to $<$, is the sequence of positive roots associated with a reduced word of $p_<(\b)$; equivalently, every initial segment of $\Phi(p_<(\b))$ with respect to $<$ is the inversion set of an element of $W$;
\item[(TR2)] for any $\a\in \Phi(p_<(\beta))$ such that $\a<_0\b$ in the root poset $(\Dp,\leq_0)$, we have  \[\supp(\b-\a)\cap \Phi(p_<(\beta))=\emptyset.\]\end{enumerate}
\end{defn}
When the condition (TR1), resp. (TR2), is satisfied for a given $\b$, we say that it holds at $\b$. A reduced word (or reduced decomposition) of $w_0$ is called {\it tangent} if the convex order it induces is tangent.
\begin{remark} Condition (TR1) is weaker than requiring $\Phi(p_<(\b))$ to be an initial section of $<$ in the sense of Subsection \ref{sec:convex}, that is, a down-set: in Example (a) below $\Phi(p_<(\a_2+\a_3))=\{\a_2,\a_2+\a_3\}$ omits $\a_1+\a_2<\a_2+\a_3$. Moreover, by Proposition \ref{p32}, $\Phi^R(p_<(\beta))=\{\b\}$.\end{remark}
\begin{prop}\label{prop:comminv}
Let $<$ and $<'$ be the convex orders induced by two reduced words of $w_0$ which are equivalent up to commutation. Then $p_<(\b)=p_{<'}(\b)$ for every $\b\in\Dp$, and $<$ is tangent if and only if $<'$ is.
\end{prop}
\begin{proof}
It suffices to treat the case in which the two words differ by one commutation, say the factor $s_ks_{k+1}$ of the first word, with $s_ks_{k+1}=s_{k+1}s_k$, is replaced by $s_{k+1}s_k$; then $\b^<_j=\b^{<'}_j$ and $x_j^<=x_j^{<'}$ for $j\neq k,k+1$, while $\b^<_k=\b^{<'}_{k+1}$, $\b^<_{k+1}=\b^{<'}_k$ and $x^<_{k+1}=x^{<'}_{k+1}$. By Lemma \ref{lem:commproj}, applied to $w=x_{k+1}^<$, we get $p_<(\b_k^<)=p(x_k^<,\b_k^<)=p(x^<_{k+1},\b_k^<)=p_{<'}(\b_k^<)$, and symmetrically for $\b^<_{k+1}$. Condition (TR2) depends only on the sets $\Phi(p_<(\b))$, which are now the same for the two orders. As for (TR1), the orders $<$ and $<'$ restricted to $\Phi(p_<(\b))$ either coincide or differ by the exchange of the consecutive roots $\b_k^<,\b^<_{k+1}$, which are orthogonal; if $t_1\cdots t_l$ is a reduced word of $p_<(\b)$ whose roots are increasing for $<$, with $\b_k^<,\b_{k+1}^<$ at positions $m,m+1$, then $(\a_{t_m},t_m(\a_{t_{m+1}}))=0$ forces $(\a_{t_m},\a_{t_{m+1}})=0$, so that $t_1\cdots t_{m-1}t_{m+1}t_mt_{m+2}\cdots t_l$ is a reduced word of $p_<(\b)$ whose roots are increasing for $<'$.
\end{proof}

\par\noindent{\bf Examples.} Let $\D$ be of type $A_3$.
\par
(a) Let $<$ be the order induced by $w_0=s_2s_1s_3s_2s_1s_3$, that is, $\a_2<\a_1+\a_2<\a_2+\a_3<\theta<\a_3<\a_1$. We verify that $<$ is tangent:
\begin{itemize}
\item $p_<(\a_2)=s_2$, for which (TR1) and (TR2) hold;
\item $p_<(\a_1+\a_2)=p(s_2s_1,\a_1+\a_2)=s_2s_1$, and again (TR1) and (TR2) hold;
\item $p_<(\a_2+\a_3)=p(s_2s_1s_3,\a_2+\a_3)=s_2s_3$. For (TR1) we have $\a_2<\a_2+\a_3$; moreover $\supp(\a_2+\a_3-\a_2)\cap \Phi(p_<(\a_2+\a_3))=\{\a_3\}\cap\{\a_2,\a_2+\a_3\}=\emptyset$, hence (TR2) holds;
\item $p_<(\theta)=p(s_2s_1s_3s_2,\theta)=s_2s_1s_3s_2$. Here (TR1) is readily satisfied, and $s_2s_1s_3s_2=d_{(s_2)}$, with $\Phi(d_{(s_2)})=[\a_2,\theta]_0$, so that (TR2) holds;
\item $p_<(\a_3)=s_3$ and $p_<(\a_1)=s_1$ clearly satisfy (TR1) and (TR2).
\end{itemize}
(b) Let $<'$ be the order induced by $w_0=s_1s_3s_2s_1s_3s_2$, that is, $\a_1<'\a_3<'\theta<'\a_2+\a_3<'\a_1+\a_2<'\a_2$. We have $p_{<'}(\theta)=p(s_1s_3s_2,\theta)=s_1s_3s_2$, so (TR1) is satisfied, but
$\supp(\theta-\a_1)=\{\a_2,\a_3\}$ and $\{\a_2,\a_3\}\cap \Phi(s_1s_3s_2)=\{\a_3\}$, hence (TR2) does not hold.
Therefore $<'$ is not a tangent convex order.

\subsection{\texorpdfstring{The condition (TR2$'$)}{The condition (TR2')}}\label{tr2prime}
The two conditions of the definition above are stated through the projection $p_<$. We
now introduce a condition which does not refer to it, and which will be the one actually
used in the second part of the paper. Throughout this subsection we fix a reduced word
${\bf i}=(i_1,\dots,i_N)$ for $w_0$ and write $\beta_1<\beta_2<\dots<\beta_N$ for the
convex order it induces, $N=|\Dp|$; we also write $w_k=s_{i_1}\cdots s_{i_k}$, so that
$\Phi(w_k)=\{\beta_1,\dots,\beta_k\}$.

\begin{defn}\label{def:tr2strong}
Let $<$ be the convex order induced by ${\bf i}$. We say that $<$ satisfies
{\rm (TR2$'$)} if for every $1\leq k\leq N$ and every $h<k$ such that
$\beta_k-\beta_h\in Q^+\setminus\{0\}$ one has
\begin{align*}
\supp(\beta_k-\beta_h)\cap\{\beta_1,\dots,\beta_{k-1}\}=\emptyset .
\end{align*}
\end{defn}

Condition (TR2$'$) is condition (TR2) with $\Phi(p_<(\beta_k))$ replaced by the
initial section $\{\beta_1,\dots,\beta_{k-1}\}$ of $<$; since
$p_<(\beta_k)\leq_R w_k$ and hence $\Phi(p_<(\beta_k))\subseteq \Phi(w_k)$, it is the
stronger of the two.
The following two facts are proved by an easy induction on the height; we leave the proofs to the reader.
\begin{lem}\label{lem:simpleprec}
Let $\gamma\in\Dp$.
\begin{enumerate}
\item[(a).] If $<$ is a convex order on $\Dp$, there exists a simple root
$\alpha\in\supp(\gamma)$ with $\alpha\leq\gamma$.
\item[(b).] If $u\in W$ and $\gamma\in \Phi(u)$, then $\Phi(u)$ contains a simple root
belonging to $\supp(\gamma)$.
\end{enumerate}
\end{lem}

\begin{prop}\label{prop:nobaddec}
The convex order $<$ induced by ${\bf i}$ satisfies {\rm (TR2$'$)} if and only if
 for no $1\leq k\leq N$ there is an identity
\begin{align}\label{eqn:baddec}
\beta_k=\sum_{p<k}a_p\beta_p+\sum_{q>k}b_q\beta_q,
\qquad a_p,b_q\in\bN,\qquad \sum_{p<k}a_p\geq2 .
\end{align}
\end{prop}
\begin{proof}
Assume (TR2$'$) and suppose \eqref{eqn:baddec} holds for some $k$. Put
$\mu=\sum_{p<k}a_p\beta_p$ and $\nu=\sum_{q>k}b_q\beta_q$. Choose $p_0<k$ with
$a_{p_0}>0$. Since $\sum_{p<k}a_p\geq2$, the element $\mu-\beta_{p_0}$ is a
nonempty sum of positive roots, whence
$\beta_k-\beta_{p_0}=(\mu-\beta_{p_0})+\nu\in Q^+\setminus\{0\}$, and there is an
index $p_1<k$ such that $\beta_{p_1}$ occurs in $\mu-\beta_{p_0}$. All the
summands lie in $Q^+$, so no cancellation of supports occurs and
$\supp(\beta_{p_1})\subseteq\supp(\beta_k-\beta_{p_0})$. By Lemma
\ref{lem:simpleprec}(a) there is a simple root $\alpha_j\in\supp(\beta_{p_1})$ with
$\alpha_j\leq\beta_{p_1}<\beta_k$; then $\alpha_j$ lies in
$\supp(\beta_k-\beta_{p_0})\cap\{\beta_1,\dots,\beta_{k-1}\}$, contradicting
(TR2$'$).
\par Conversely, assume that (TR2$'$) fails, say
$\delta:=\beta_k-\beta_h\in Q^+\setminus\{0\}$ with $h<k$ and
$\alpha_{j_0}\in\supp(\delta)\cap\{\beta_1,\dots,\beta_{k-1}\}$. Then $\beta_k$ is
not simple: were it simple, no positive root would satisfy
$\beta_k-\beta_h\in Q^+\setminus\{0\}$. Writing $\delta=\sum_j c_j\alpha_j$ with
$c_j\in\bN$ and $c_{j_0}\geq1$, and observing that every simple root is one of the
$\beta_r$, with $r\neq k$ because $\beta_k$ is not simple, we obtain
\begin{align*}
\beta_k=\beta_h+\sum_j c_j\alpha_j ,
\end{align*}
a decomposition of $\beta_k$ into positive roots all different from $\beta_k$, at
least two of which, namely $\beta_h$ and $\alpha_{j_0}$, precede $\beta_k$.
Sorting the summands according to their position with respect to $\beta_k$ yields
an identity of the shape \eqref{eqn:baddec}.
\end{proof}

We call an identity of the shape \eqref{eqn:baddec} a {\it forbidden decomposition}
of $\beta_k$: thus Proposition \ref{prop:nobaddec} says that $<$ satisfies {\rm (TR2$'$)}
if and only if no $\beta_k$ admits a forbidden decomposition. In Subsection
\ref{sec:reduction} the same identity will be recorded by the pair of exponent vectors
$\bfa,\bfb\in\bN^N$, for which, with the notation of Subsection \ref{sec:pbwnot}, it reads $|\bfa|+|\bfb|=\beta_k$,
$\supp(\bfa)\subseteq[1,k-1]$, $\supp(\bfb)\subseteq[k+1,N]$ and $\sum_j a_j\geq2$.
\par
A subset $A\subseteq\Dp$ is {\it abelian}, or {\it sum free}, if no two of its elements
add up to a root, that is if $(A+A)\cap\D=\emptyset$; equivalently, $\bigoplus_{\a\in A}\g_\a$
is an abelian subalgebra of $\mathfrak n^+$.

\begin{prop}\label{prop:Bk}
Assume {\rm (TR2$'$)} and put $B_k=\Phi(w_k)\cap(\beta_k-Q^+)$. Then $B_k$ is abelian and biconvex, and
$\Phi(p_<(\beta_k))\subseteq B_k\subseteq\beta_k-Q^+$.
\end{prop}
\begin{proof}
Let $\gamma,\delta\in B_k$, i.e. $\gamma,\delta$ precede $\beta_k$ or are equal to it and
$\gamma,\delta\leq_0\beta_k$, and suppose $\eta:=\gamma+\delta\in\Dp$. By convexity $\eta$ lies between $\gamma$ and
$\delta$; relabeling if necessary, $\gamma<\eta<\delta$, so that $\gamma<\beta_k$. By Lemma \ref{lem:simpleprec}(a) there is a simple root
$\a_r\in\supp(\gamma)$ with $\a_r\leq\gamma$, so that $\a_r<\eta$ and $\a_r<\beta_k$.
Since $\gamma$ precedes $\beta_k$ and $\beta_k-\gamma\in Q^+\setminus\{0\}$,
condition (TR2$'$) at $\beta_k$ forbids any simple root of $\supp(\beta_k-\gamma)$ to
precede $\beta_k$; as $\a_r$ does, $\a_r\notin\supp(\beta_k-\gamma)$, that is
$c_r(\gamma)=c_r(\beta_k)$, where $c_r$ denotes the coefficient of $\a_r$. The same
argument with $\delta$ gives $c_r(\delta)=c_r(\beta_k)$, trivially so if $\delta=\beta_k$. Hence
$c_r(\eta)=c_r(\gamma)+c_r(\delta)=2c_r(\beta_k)\geq2$,
because $c_r(\gamma)\geq1$. Now apply (TR2$'$) at $\eta$ with the preceding root $\a_r$:
one has $\eta-\a_r\in Q^+\setminus\{0\}$ and $c_r(\eta-\a_r)=c_r(\eta)-1\geq1$, so that
$\a_r$ is a simple root of $\supp(\eta-\a_r)$ preceding $\eta$, contrary to (TR2$'$). Hence $B_k$ is abelian, and in
particular closed.
\par It is co-closed: if $\mu+\nu\in B_k$ with
$\mu,\nu\in\Dp$, then one of $\mu,\nu$ lies in $\Phi(w_k)$, the latter being biconvex, say $\mu$; and
$\beta_k-\mu=(\beta_k-\mu-\nu)+\nu\in Q^+$, so that $\mu\in B_k$. Hence $B_k$ is biconvex, and the
minimality of $\Phi(p_<(\beta_k))$ among the biconvex subsets of $\Phi(w_k)$ containing
$\beta_k$ (Proposition \ref{p32}) gives $\Phi(p_<(\beta_k))\subseteq B_k$.
\end{proof}
\begin{lem}\label{lem:tr1reform}
Condition {\rm (TR1)} holds at $\beta_k$ if and only if $\Phi(p_<(\beta_k))\cap \Phi(w_j)$ is
biconvex for every $j$. Moreover {\rm (TR1)} fails at $\beta_k$ if and only if there are
$\mu,\nu\in\Dp$ with
\begin{align}\label{eqn:tr1config}
\nu<\mu+\nu<\mu\leq\beta_k,\qquad \mu,\,\mu+\nu\in \Phi(p_<(\beta_k)),\qquad \nu\notin \Phi(p_<(\beta_k)).
\end{align}
\end{lem}
\begin{proof}
Every initial segment of $\Phi(p_<(\beta_k))$ with respect to $<$ is of the form
$\Phi(p_<(\beta_k))\cap \Phi(w_j)$, whence the first assertion. Such an intersection is closed,
being an intersection of closed sets, so it can fail to be biconvex only if it is not co-closed. If $\zeta=\mu+\nu$ lies in the intersection while neither $\mu$ nor $\nu$ does, then,
$\Phi(p_<(\beta_k))$ and $\Phi(w_j)$ being biconvex, one of $\mu,\nu$ lies in $\Phi(p_<(\beta_k))$ and
one of them lies in $\Phi(w_j)$; these cannot be the same root, so that, say,
$\mu\in \Phi(p_<(\beta_k))\setminus \Phi(w_j)$ and $\nu\in \Phi(w_j)\setminus \Phi(p_<(\beta_k))$. Since $\zeta\in\Phi(w_j)$ and $\mu\notin\Phi(w_j)$, we have $\zeta<\mu$, and the convexity of $<$, which places $\zeta$ between
$\mu$ and $\nu$, gives $\nu<\zeta<\mu$; finally $\mu\leq\beta_k$, as $\mu\in\Phi(p_<(\beta_k))\subseteq\Phi(w_k)$. This is \eqref{eqn:tr1config}. Conversely, given $\mu,\nu$ as in \eqref{eqn:tr1config}, let $\mu+\nu=\beta_j$: then $\mu+\nu$ lies in $\Phi(p_<(\beta_k))\cap\Phi(w_j)$, while $\mu$ and $\nu$ do not, so that this set is not co-closed.
\end{proof}

\begin{thm}\label{thm:tr2ptangent}
A convex order satisfying {\rm (TR2$'$)} is tangent. Consequently, if ${\bf i}$ is not a
tangent reduced decomposition, then some $\beta_k$ admits a forbidden decomposition.
\end{thm}
\begin{proof}
Let $\a\in \Phi(p_<(\beta_k))$ with $\a<_0\beta_k$. Then $\beta_k$ is not simple, and $\a=\beta_h$ for some $h<k$,
because $\Phi(p_<(\beta_k))\subseteq \Phi(w_k)$ and $\a\neq\beta_k$; moreover $\beta_k\notin\supp(\beta_k-\a)$. Hence
$\supp(\beta_k-\a)\cap \Phi(p_<(\beta_k))\subseteq\supp(\beta_k-\a)\cap\{\beta_1,\dots,\beta_{k-1}\}=\emptyset$, and
(TR2) holds. If (TR1) failed at some $\beta_k$, Lemma \ref{lem:tr1reform} would provide $\mu,\nu\in\Dp$ with
$\nu<\mu+\nu<\mu\leq\beta_k$ and $\mu,\mu+\nu\in \Phi(p_<(\beta_k))\subseteq B_k$ (Proposition \ref{prop:Bk}). As
$\mu+\nu\leq_0\beta_k$, we have $\mu\neq\beta_k$; so $\mu$ and $\nu$ precede $\beta_k$, and $\nu\leq_0\mu+\nu\leq_0\beta_k$.
Thus $\mu,\nu\in B_k$ and $\mu+\nu\in\Dp$, contradicting the fact that $B_k$ is abelian. The last assertion follows from
Proposition \ref{prop:nobaddec}.
\end{proof}

We record two elementary facts on projections, which will be used repeatedly; the first of them is a direct
consequence of the biconvexity of projections, and neither of the two uses {\rm (TR2)}.

\begin{lem}\label{lem:proj}
Let $\beta=\beta_k$.
\begin{enumerate}
\item[(a).] If $\beta=\mu+\nu$ with $\mu,\nu\in\Dp$ and $\mu<\beta<\nu$, then
$\mu\in \Phi(p_<(\beta))$.
\item[(b).] If {\rm (TR1)} holds at $\beta$, then $\Phi(p_<(\gamma))\subseteq \Phi(p_<(\beta))$
for every $\gamma\in \Phi(p_<(\beta))$.
\end{enumerate}
\end{lem}
\begin{proof}
(a). The set $\Phi(p_<(\beta))$ is biconvex, so $\Dp\setminus \Phi(p_<(\beta))$ is closed. As
$\nu>\beta$ we have $\nu\notin \Phi(w_k)\supseteq \Phi(p_<(\beta))$; if $\mu$ were outside
$\Phi(p_<(\beta))$ as well, closedness of the complement would give
$\beta=\mu+\nu\notin \Phi(p_<(\beta))$, which is false.
\par (b). Write $\gamma=\beta_j$ with $j\leq k$. By Lemma \ref{lem:tr1reform} the set
$\Phi(p_<(\beta))\cap \Phi(w_j)$ is biconvex; it is contained in $\Phi(w_j)$ and contains $\gamma$.
By Proposition \ref{p32}, $\Phi(p_<(\gamma))$ is the smallest biconvex subset of $\Phi(w_j)$
containing $\gamma$, whence the claim.
\end{proof}

In the language of the weak order, Lemma \ref{lem:proj}(b) says that, under (TR1) at
$\beta_k$, the join-irreducible element $p_<(\gamma)$ lies below $p_<(\beta_k)$ whenever
$\gamma$ is an inversion of the latter.

\subsection{Tangent convex orders and longest coset representatives}
\begin{prop}\label{p41} Let $W$ be a finite irreducible Weyl group, and let $a_1,\ldots,a_n$ be an enumeration of $S$. Consider
\[W_{(a_1)}=W_{S\setminus \{a_1\}},\ W_{(a_1,a_2)}=W_{S\setminus \{a_1,a_2\}},\ \ldots,\ W_{(a_1,\ldots,a_n)}=W_{S\setminus \{a_1,a_2,\ldots,a_{n}\}}=\{e\}.\]
Assume that, for all $i$, $\a_{a_i}$ is cominuscule in $W_{(a_1,\ldots,a_{i-1})}$, that is, in the irreducible component of $W_{(a_1,\ldots,a_{i-1})}$ containing $a_i$, where $W_{(a_1,\ldots,a_0)}:=W$. We consider the following reduced words:
\begin{align*}
d_{(a_1)}&=s_1\cdots s_{k_1},\\
d^{(a_1)}_{(a_1,a_2)}&=d^{S\setminus\{a_1\}}_{S\setminus \{a_1,a_2\}}=s_{k_1+1}\cdots s_{k_2},\\
&\vdots\\
d^{(a_1,\ldots,a_{n-2})}_{(a_1,\ldots,a_{n-1})}&=d^{\{a_{n-1},a_n\}}_{\{a_n\}}=s_{k_{n-2}+1}\cdots s_{k_{n-1}},\\
d^{(a_1,\ldots,a_{n-1})}_{(a_1,\ldots,a_{n})}&=a_n=s_{k_{n}}.\\
\end{align*}
Then $s_{k_i}=a_i$ and $k_n=N,\, s_{k_1}=w_0s_1w_0,\,s_{k_i+1}=w_{0,(a_1,\ldots,a_i)}s_{k_{i+1}} w_{0,(a_1,\ldots,a_i)}$ and
\begin{equation}\label{x}w_0=s_1\cdots s_N=d_{(a_1)}d^{(a_1)}_{(a_1,a_2)}\cdots d^{(a_1,\ldots,a_{n-1})}_{(a_1,\ldots,a_{n})}\end{equation}
is a reduced word for which the associated convex order is a tangent convex order.
\end{prop}
We call such a convex order a {\it canonical tangent convex order}.
The proposition follows from the next lemma and induction on $|S|$; see Remark \ref{rem:reducible} for the parabolic subgroups which are not irreducible.
\begin{remark}\label{rem:reducible}
The definitions of Subsections \ref{sec:convex}--\ref{tr2prime}, together with the results proved there, make sense for an arbitrary finite Weyl group, since irreducibility was never used. If $W=W_1\times W_2$, with $\Dp=\Dp_1\sqcup\Dp_2$, a convex order $<$ on $\Dp$ restricts to convex orders $<_1$, $<_2$ on $\Dp_1$, $\Dp_2$, and its reduced word is a shuffle of reduced words inducing $<_1$ and $<_2$, hence it is equivalent up to commutation to their concatenation. Moreover $p_<(\b)=p_{<_1}(\b)\in W_1$ for $\b\in\Dp_1$, and similarly for $\Dp_2$, so that $<$ is tangent (resp. satisfies (TR2$'$)) if and only if both $<_1$ and $<_2$ are tangent (resp. satisfy (TR2$'$)); similarly, the reduced word of $<$ is equivalent up to commutation to a canonical one if and only if those of $<_1$ and $<_2$ are. This is how the inductions on $|S|$ in this section are to be understood when a parabolic subgroup $W_{(s)}$ is not irreducible: one applies the inductive hypothesis to its irreducible components.
\end{remark}
\begin{lem}\label{l42} Keep the same notation as in Proposition \ref{p41}.
\begin{enumerate}
\item[(a).] $-d_{(a_1)}(\a_{a_1})=\theta$, i.e. $\Phi^R(d_{(a_1)})=\{\theta\}$ and $\a_{s_1}=-w_0(\a_{a_1})$ is cominuscule.
\item[(b).] $\theta$ satisfies (TR1) and (TR2).
\item[(c).] For all $\b<\theta$, $\b$ satisfies (TR1) and (TR2).
\item[(d).] If $<_{(a_1)}$ is a tangent convex order on $\D^+_{(a_1)}$ given by the reduced word $w_{0,(a_1)}=r_{k_1+1}\cdots r_N$, then
$s_1\cdots s_{k_1}r_{k_1+1}\cdots r_N$ is a reduced word associated to a tangent convex order on $\Dp$.
\end{enumerate}
\end{lem}
\begin{proof} (a). Let $x=d^{(a_1)}_{(a_1,a_2)}\cdots d^{(a_1,\ldots,a_{n-1})}_{(a_1,\ldots,a_{n})}$ (cf. \eqref{x}). Since  $x\in W_{(a_1)}=W_{S\setminus\{a_1\}}$ and the product
$w_0=d_{(a_1)}x$ is reduced, one must have $\theta\in \Phi(d_{(a_1)})$. We know by Proposition \ref{p21} that $\Phi(d_{(a_1)})=[-w_0(\a_{a_1}),\theta]_0$.
Since $\Pi\cap [-w_0(\a_{a_1}),\theta]_0=\{ -w_0(\a_{a_1})\}$ and $\a_{s_1}\in \Phi(d_{(a_1)})$, we have $\a_{s_1}=-w_0(\a_{a_1})$. Since $\a_{a_1}$ is cominuscule in $\D$, also $\a_{s_1}$ is cominuscule. Now let $\b= -d_{(a_1)}(\a_{a_1})$. Since $D_R(d_{(a_1)})=\{a_1\}$ by Proposition \ref{p13}(5), we have $\Phi^R(d_{(a_1)})=\{\b\}$. Now
$-d_{(a_1)}^{-1}(\theta)=-w_{0,(a_1)}w_0(\theta)=w_{0,(a_1)}(\theta)=\a_{a_1}$, where the last equality follows, e.g., from \cite{IM}, $\a_{a_1}$ being cominuscule; hence $d_{(a_1)}(\a_{a_1})=-\theta$ and $\b=\theta$.

\par (b). We have $p_<(\theta)=d_{(a_1)}$ since $\Phi^R(d_{(a_1)})=\{\theta\}$ by (a) and Proposition \ref{p32}. Now $\theta$ satisfies (TR1) by construction, and it satisfies (TR2) because $\a_{s_1}$ is
cominuscule: every $\b\in[\a_{s_1},\theta]_0=\Phi(d_{(a_1)})$ has $c_{s_1}(\b)=c_{s_1}(\theta)=1$,
so that $\supp(\theta-\b)\subseteq \Pi\setminus\{\a_{s_1}\}$, while every element of
$\Phi(d_{(a_1)})$ has $\a_{s_1}$ in its support. So $\Phi(d_{(a_1)})\cap \supp(\theta-\b)=\emptyset$ for all $\b\in \Phi(d_{(a_1)})$.
\par
(c). Let $\b<\theta$, so that $\b\in \Phi(d_{(a_1)})$. We first show that $\b$ satisfies (TR1). Let $x_i=s_1\cdots s_i$, where $\b=s_1\cdots s_{i-1}(\a_{s_i})$, so that $i\leq k_1$. Then $p_<(\b)\leq _R x_i$, i.e. $p_<(\b)u=x_i$, in reduced form, for some $u\in W$. Since $d_{(a_1)}$ is fully commutative by Proposition \ref{p22}, so is its prefix $x_i$. Therefore the orders on $\Phi(x_i)$ induced by the reduced words of $x_i$ are exactly the linear extensions of the heap of $x_i$ \cite[\S 3]{St}; the heap of the prefix $p_<(\b)$ is an order ideal of it, and the restriction of a linear extension to an order ideal is a linear extension. Hence the restriction of $<$ to $\Phi(p_<(\b))$ is induced by a reduced word of $p_<(\b)$, which is (TR1) at $\b$.
Condition (TR2) at $\b$ follows as in part (b): $\Phi(p_<(\b))\subseteq \Phi(d_{(a_1)})=[\a_{s_1},\theta]_0$,
so every $\a\in \Phi(p_<(\b))$ has $c_{s_1}(\a)=1=c_{s_1}(\b)$ and therefore
$\supp(\b-\a)\subseteq\Pi\setminus\{\a_{s_1}\}$, which is disjoint from $\Phi(d_{(a_1)})$.\par
(d). By (b) and (c), it remains to check (TR1) and (TR2) at the roots $\b\notin\Phi(d)$, where $d=d_{(a_1)}$. Such a root is $\b=\b_{k_1+j}=s_1\cdots s_{k_1}r_{k_1+1}\cdots r_{k_1+j-1}(\a_{r_{k_1+j}})=d(\b')$ for some $1\leq j\leq N-k_1$, where $\b'=r_{k_1+1}\cdots r_{k_1+j-1}(\a_{r_{k_1+j}})$ is the $j$-th root for the order $<_{(a_1)}$. Recall that $d(\D^+_{(a_1)})=\Dp_{(w_0a_1w_0)}=\Dp_{(s_1)}$ by Proposition \ref{p13}; more precisely, $d$ maps $\Pi_{(a_1)}$ onto $\Pi_{(s_1)}$, hence it is an isomorphism of root systems from $\D_{(a_1)}$ onto $\D_{(s_1)}$ which preserves the root posets and the supports, and $dW_{(a_1)}d^{-1}=W_{(s_1)}$. By the transfer identity \eqref{toprove}, established in the proof of Theorem \ref{t43} below by an argument which applies verbatim here, $\Phi(p_<(\b))=d\bigl(\Phi(p_{<_{(a_1)}}(\b'))\bigr)$. If $t_1\cdots t_l$ is a reduced word of $p_{<_{(a_1)}}(\b')$ whose roots are listed increasingly by $<_{(a_1)}$, then $(dt_1d^{-1})\cdots(dt_ld^{-1})$ is a reduced word of $p_<(\b)=dp_{<_{(a_1)}}(\b')d^{-1}$ whose roots, namely the images under $d$ of the previous ones, are listed increasingly by $<$; so (TR1) at $\b'$ implies (TR1) at $\b$. Similarly, (TR2) at $\b$ is the image under $d$ of (TR2) at $\b'$.
\end{proof}
\subsection{Classification of tangent convex orders}\label{sec:classif}
In this subsection we prove the following theorem.
\begin{thm}\label{t43} Let $<$ be a convex order with reduced word $w_0=s_1\cdots s_N$. Then $<$ is a tangent convex order if and only if $s_1\cdots s_N$ is equivalent up to commutation to a reduced word which starts with $s_1$ and induces a canonical tangent convex order. Moreover, in this case $\a_{s_1}$ is cominuscule.
\end{thm}
In particular there are no tangent convex orders in types $E_8$, $F_4$ and $G_2$, which have no
cominuscule simple root. The rest of this subsection is devoted to the proof of the theorem.
\par Let $1\leq i\leq N$ be such that $\theta=\b_i^<$. If (TR1) holds at $\theta$, we let $p_<(\theta)=r_1\cdots r_p$ be the reduced word whose sequence of roots is $\Phi(p_<(\theta))$ listed in increasing order with respect to $<$. Since $p_<(\theta)\leq _R x_i$ by definition, there exist $r_{p+1},\ldots,r_i\in S$ such that
$x_i=r_1\cdots r_i$ is a reduced word; a suitable choice of these letters will be made in Lemma \ref{l46}.\par

\begin{lem}\label{l44} Let $<$ be a tangent convex order. Then
\[p_<(\theta)=d_{(w_0r_1w_0)},\qquad \Phi(p_<(\theta))=[\a_{r_1},\theta]_0,\]
and both $\a_{r_1}$ and $-w_0(\a_{r_1})$ are cominuscule.
\end{lem}
\begin{proof} The rank one case being trivial, we may assume that $\theta$ is not simple. By Proposition \ref{p21} we just have to show that $\Phi(p_<(\theta))=[\a_{r_1},\theta]_0$. Let $\b\in \Phi(p_<(\theta))$, so that $\b\leq_0\theta$, $\theta$ being the highest root. Assume $\b\notin [\a_{r_1},\theta]_0$; then $\b\neq\theta$ and
$\a_{r_1}\notin \supp(\b)$. So $\a_{r_1}\in \supp(\theta-\b)$, and by (TR2) we have $\a_{r_1}\notin \Phi(p_<(\theta))$, which contradicts $\Phi(p_<(\theta))=\{\a_{r_1},r_1(\a_{r_2}),\ldots\}$. So $\b \geq_0 \a_{r_1}$, and
$\Phi(p_<(\theta))\subseteq [\a_{r_1},\theta]_0$. Note now that, since $\a_{r_1}, \theta \in \Phi(p_<(\theta))$ and $\a_{r_1}<_0\theta$, (TR2) gives $\supp(\theta-\a_{r_1})\cap \Phi(p_<(\theta))=\emptyset$. Hence
$c_{r_1}(\theta)=1$, since otherwise $\a_{r_1}\in \supp(\theta-\a_{r_1})\cap \Phi(p_<(\theta))$; therefore $\a_{r_1}$, and hence also $-w_0(\a_{r_1})$, is cominuscule. Moreover $\supp(\theta-\a_{r_1})=\Pi_{S\setminus \{r_1\}}$, as $\theta$ has full support. Now let $\b\in [\a_{r_1},\theta]_0\setminus \Phi(p_<(\theta))$. Since
$\a_{r_1}\in \Phi(p_<(\theta))$ and $\supp(\theta-\a_{r_1})=\Pi_{S\setminus \{r_1\}}$, condition (TR2) gives again $\Pi_{S\setminus \{r_1\}}\subseteq \Dp\setminus  \Phi(p_<(\theta))$. As $c_{r_1}(\b)=1=c_{r_1}(\theta)$, we have $\theta=\b+\sum_{t\ne r_1}n_t\a_t$ with $n_t\in\bN$, so that $\theta\in
\cone_\D(\b,\a_t; t\ne r_1)\subseteq \Dp\setminus \Phi(p_<(\theta))$, since $\Dp\setminus \Phi(p_<(\theta))$ is biconvex by Proposition \ref{p11}. So $\theta\notin \Phi(p_<(\theta))$, which is a contradiction.
\end{proof}
\begin{lem}\label{l45} We have $w_0r_1w_0=r_p=s_i$. So $p_<(\theta)=d_{(s_i)}.$
\end{lem}
\begin{proof} We have $w_0r_1w_0=r_p$ by Proposition \ref{p13}(5), since $D_R(d_{(w_0r_1w_0)})=\{w_0r_1w_0\}$ and $d_{(w_0r_1w_0)}=r_1\cdots r_p$ is reduced. Now
$d_{(r_p)}=p_<(\theta)\leq_R  x_i$ by definition. So by Proposition \ref{p16}, $x_i=d_{(r_p)}v$ in reduced form with $v\in W_{(r_p)}$. Now, as in the proof  of Lemma \ref{l42},  we have, by \cite{IM},
\[-\a_{r_p}=d^{-1}_{(r_p)}(\theta)=vx_i^{-1}(\theta)=v(-\a_{s_i}),\]
since $\theta=s_1\cdots s_{i-1}(\a_{s_i})=-x_i(\a_{s_i})$.
Then $\a_{r_p}=v(\a_{s_i})$. Since $v\in W_{(r_p)}=W_{S\setminus\{r_p\}}$, one must have $\a_{s_i}\notin \Pi_{S\setminus\{r_p\}}$, so $s_i=r_p$.
\end{proof}
\begin{lem}\label{l46} Let $<$ be a tangent convex order. For a suitable choice of $r_{p+1},\ldots,r_i$, the reduced words $r_1\cdots r_i$ and $s_1\cdots s_i$ are equivalent up to commutation. Moreover $r_1=s_1$.
\end{lem}
\begin{proof} If $p=i$, i.e., $p_<(\theta)=x_i=s_1\cdots s_i$, then the two words have the same sequence of roots, hence they coincide, since a reduced word is determined by its sequence of roots.
Assume $p<i$. We have $s_1\cdots s_{i-1}(\a_{s_i})=\b_i^<=\theta$. Since $s_i\in X_{(s_i)}$, while $x_i=d_{(s_i)}v$ with $v\neq e$ by Lemma \ref{l45}, so that $x_i\notin X_{(s_i)}$, there is $1\leq k\leq i-1$ such that
$u=s_{k+1}\cdots s_i\in X_{(s_i)}$ and $s_ku\notin X_{(s_i)}$. Since $s_ku$ is a reduced product, Deodhar's Lemma gives $s_ku=ut$ with $t\in S\setminus\{s_i\}$; note that $u(\a_t)=\a_{s_k}$.
We show that $s_ks_j=s_js_k$ for all $k+1\leq j\leq i$,
that is, $(\a_{s_k}, \a_{s_j})=0$ for all $k+1\leq j\leq i$.
\par
We first show that $(\b_k^<,\b_{k+1}^<)=0$. Indeed, we know that
\[s_{\b_k^<}(\b_{k+1}^<)=\b_{k+1}^<-\langle \b_{k+1}^<,(\b_{k}^<)^\vee\rangle\b_{k}^<.\]
Note that \[s_{\b_{k}^<}(\b_{k+1}^<)=s_1\cdots s_{k-1} s_k s_{k-1} s_{k-2}\cdots s_1 s_1 \cdots s_k(\a_{s_{k+1}})=s_1\cdots s_{k-1}(\a_{s_{k+1}}) \in \Dp\]
since $s_1\cdots s_{k-1}s_{k+1}$ is reduced. Indeed, $x_i=s_1\cdots s_{k}\cdots s_{i}= s_1\cdots \widehat{s_{k}}\cdots s_{i}t$ which is reduced, so that the prefix
$s_1\cdots s_{k-1}s_{k+1}$ is reduced too. Also note that $(\b_k^<,\b_{k+1}^<)=-(\a_{s_k},\a_{s_{k+1}})\geq 0$.\par
Assume $(\b_{k+1}^<,\b_{k}^<)>0$, so that $m:=\langle \b_{k+1}^<,(\b_{k}^<)^\vee\rangle\geq 1$. Then in the decomposition
\[\b_{k+1}^<= s_{\b_{k}^<}(\b_{k+1}^<)+m\,\b_{k}^<\]
the summands on the right-hand side are positive roots and
$\b_{k}^<<\b_{k+1}^<<s_{\b_{k}^<}(\b_{k+1}^<)$ by Proposition \ref{pp}. Let $z=p_<(\b_{k+1}^<)$. By biconvexity of $\Phi(z)$, either
$\b_k^<$ or $s_{\b_{k}^<}(\b_{k+1}^<)$ lies in $\Phi(z)$, since $\b_{k+1}^<\in \Phi(z)$. But $s_{\b_k^<}(\b^<_{k+1})$ follows $\b^<_{k+1}$, hence does not lie in $\Phi(x_{k+1})\supseteq \Phi(z)$; so $\b_k^<\in \Phi(z)$ (cf. Lemma \ref{lem:proj}(a)). Now $\b_{k+1}^<\in \Phi(d_{(s_i)})$ by Proposition \ref{p16} and Lemma \ref{l45}.  Since $<$ is tangent, (TR1) holds at $\theta$, and $d_{(s_i)}=p_<(\theta)$ by Lemma \ref{l45}; so Lemma \ref{lem:proj}(b) gives $\Phi(z)=\Phi(p_<(\b^<_{k+1}))\subseteq \Phi(p_<(\theta))=\Phi(d_{(s_i)})$, that is $z\leq_R d_{(s_i)}$, and therefore
$\b_k^<\in \Phi(z)\subseteq \Phi(d_{(s_i)})$. But $y:=s_1\cdots \widehat{s_{k}}\cdots s_{i}=x_it$ lies in $x_iW_{(s_i)}=d_{(s_i)}W_{(s_i)}$, so that $d_{(s_i)}\leq_R y$, $d_{(s_i)}$ being the minimal representative of this coset; moreover $\Phi(y)=\Phi(x_i)\setminus\{-x_i(\a_t)\}=\Phi(x_i)\setminus\{\b^<_k\}$, as $-x_i(\a_t)=s_1\cdots s_{k-1}u(\a_t)=\b^<_k$.
So $\Phi(d_{(s_i)})\subseteq \Phi(x_i)\setminus\{\b^<_k\}$, a contradiction. We conclude that $(\b_k^<,\b_{k+1}^<)=0$ and
$s_ks_{k+1}=s_{k+1}s_k$. So $s_1\cdots s_ks_{k+1}\cdots s_i$ is equivalent up to commutation to $s_1\cdots s_{k+1}s_k\cdots s_i$. Now we repeat the argument with the convex order obtained by switching $\b_k^<$ and $\b_{k+1}^<$, which is again tangent by Proposition \ref{prop:comminv}; note that $x_i$, $\theta$ and $p_<(\theta)=d_{(s_i)}$ are unchanged, and that $s_k$ now occupies the position $k+1$, with $s_ks_{k+2}\cdots s_i=s_{k+2}\cdots s_it$. We obtain that $s_1\cdots s_k\cdots s_i$ and $s_1\cdots \widehat{s_k}\cdots s_is_k$ are equivalent up to commutation, and that $\a_{s_k}$ is orthogonal to $\a_{s_{k+1}},\dots,\a_{s_i}$. Now $\theta$ is the last root of the reduced word $s_1\cdots \widehat{s_k}\cdots s_i$, of the element $x_is_k$, whose projection along $\theta$ is again $d_{(s_i)}$, and we conclude by induction on $\ell(x_i)-\ell(d_{(s_i)})$.\par
What we obtain is a sequence of indices $k_1<\dots<k_q$, with $k_q=k$, such that
\[x_i=s_1\cdots s_i\sim\underbrace{s_1\cdots \widehat{s_{k_1}}\cdots \widehat{s_{k_2}}\cdots \widehat{s_{k_q}}\cdots s_i}_{=r_1\cdots r_p}\ \underbrace{s_{k_1}s_{k_2}\cdots s_{k_q}}_{=v\in W_{(s_i)}},\]
where the first underbraced word is a reduced word of $d_{(s_i)}$ whose sequence of roots is $\Phi(d_{(s_i)})$ listed increasingly with respect to $<$ (commutations do not change the root attached to a letter), hence it coincides with $r_1\cdots r_p$; we choose $r_{p+1}\cdots r_i=s_{k_1}\cdots s_{k_q}$. Finally, if $k_1=1$, then, by the argument above applied at the step of the induction in which $s_1$ is moved, $\a_{s_1}$ is orthogonal to $\a_{s_j}$ for every letter $s_j$ of the word $s_2\cdots \widehat{s_{k_2}}\cdots \widehat{s_{k_q}}\cdots s_i$, which is then a reduced word of $d_{(s_i)}$; in particular none of these letters equals $s_1$. Hence $d_{(s_i)}\in W_{S\setminus\{s_1\}}$, which is a contradiction because $\Phi(d_{(s_i)})\ni\theta$ and $\theta$ has full support. Thus $k_1>1$, so that $s_1$ is the first letter of $r_1\cdots r_p$, i.e. $r_1=s_1$.
\end{proof}
\begin{proof}[Proof of Theorem \ref{t43}]
The ``if'' part follows from Proposition \ref{p41} and Proposition \ref{prop:comminv}. For the ``only if'' part, let $<$ be tangent. By Lemma \ref{l46},
$w_0=s_1\cdots s_i s_{i+1}\cdots s_N$ is equivalent up to commutation to
\[\underbrace{r_1\cdots r_p}_{=d_{(s_i)}} \underbrace{r_{p+1}\cdots r_i s_{i+1}\cdots s_N}_{=w_{0,(s_i)}\in W_{(s_i)}}=w_0.\]
The associated convex order $<'$ is tangent by Proposition \ref{prop:comminv}. Let now $<_1$ be the convex order on $\Dp_{S\setminus\{s_i\}}$ induced by the reduced word $r_{p+1}\cdots r_i s_{i+1}\cdots s_N$ of $w_{0,(s_i)}$. It is enough to show that $<_1$ is tangent, and to proceed by induction on $|S|$ (see Remark \ref{rem:reducible}).\par
To simplify notation, we write $<$ for $<'$, write $d=d_{(s_i)}$, and index the roots of $<_1$ so that $\b_j^<=d(\b_j^{<_1})$ for $j\ge p+1$; similarly $x_j^<=d\,x_j^{<_1}$. Then (TR1) and (TR2) for $<_1$ follow from those for $<$, as in the proof of Lemma \ref{l42}(d), once we show that
\begin{equation}\label{toprove}\Phi(p_<(\b_j^<))=d_{(s_i)}\left(\Phi(p_{<_1}(\b_j^{<_1}))\right),\quad\text{ equivalently }\quad
p_<(\b_j^<)=d_{(s_i)}\,p_{<_1}(\b_j^{<_1})\,d_{(s_i)}^{-1}.\end{equation}
Let $A\subseteq\Dp_{(s_i)}$
be biconvex in $\D_{(s_i)}$. Then $d(A)$ is biconvex in $\Dp$: it is closed, because
$d(\mu)+d(\nu)=d(\mu+\nu)\in\Dp$ forces $\mu+\nu\in\Dp_{(s_i)}$ and hence $\mu+\nu\in A$; and its complement is
closed, because if $\sigma+\tau\in d(A)$ with $\sigma,\tau\in\Dp$, then $\sigma,\tau\notin \Phi(d)$, as every element of $\Phi(d)$
has $c_{s_1}=1$ while $c_{s_1}$ vanishes on $d(\Dp_{(s_i)})=\Dp_{(s_1)}$, so that $\sigma=d(\sigma')$ and $\tau=d(\tau')$ with
$\sigma',\tau'\in\Dp_{(s_i)}$ and $\sigma'+\tau'\in A$, whence $\sigma'\in A$ or $\tau'\in A$.
Apply this to $A=\Phi(p_{<_1}(\b_j^{<_1}))$: the set $d(A)$ is a biconvex subset of
$\Phi(x_j)=\Phi(d)\sqcup d(\Phi(x_j^{<_1}))$ containing $\b_j^<$, so the minimality of $\Phi(p_<(\b_j^<))$ among
such sets gives $\Phi(p_<(\b_j^<))\subseteq d(A)$. In particular $\Phi(p_<(\b_j^<))\subseteq d(\Dp_{(s_i)})$,
so that $d^{-1}\left(\Phi(p_<(\b_j^<))\right)$ is a biconvex subset of $\Phi(x_j^{<_1})$ containing
$\b_j^{<_1}$, and the minimality of $\Phi(p_{<_1}(\b_j^{<_1}))$ gives the reverse inclusion. This proves
\eqref{toprove}.
\end{proof}

\subsection{\texorpdfstring{Tangency and (TR2$'$)}{Tangency and (TR2')}}\label{tangvstr2p}

\begin{prop}\label{prop:canonicalTR2p}
Every canonical tangent convex order satisfies {\rm (TR2$'$)}.
\end{prop}
\begin{proof}
We argue by induction on $|S|$, the statement being empty when $S=\emptyset$. Let $<$ be
the convex order attached to the reduced word \eqref{x}, and set $d=d_{(a_1)}$. By Lemma
\ref{l42}(a) the simple root $\a_{s_1}$ is cominuscule, and by Proposition \ref{p13}(4) the
initial section $\Phi(d)$ of $<$ is
\begin{align*}
\Phi(d)=\Dp\setminus\Dp_{(s_1)}=\{\gamma\in\Dp\mid c_{s_1}(\gamma)>0\},
\end{align*}
$c_{s_1}(\gamma)$ denoting the coefficient of $\a_{s_1}$ in $\gamma$; the complementary
final section is $\Dp_{(s_1)}$, on which $<$ restricts to a canonical tangent convex
order of $W_{(s_1)}$. Since $\a_{s_1}$ is cominuscule, $c_{s_1}$ takes only the values
$0$ and $1$ on $\Dp$, so that $c_{s_1}\equiv1$ on $\Phi(d)$.
\par Suppose now that some $\beta_k$ admitted a forbidden decomposition
\begin{align*}
\beta_k=\sum_{p<k}a_p\beta_p+\sum_{q>k}b_q\beta_q,\qquad \sum_{p<k}a_p\geq2 .
\end{align*}
If $\beta_k\in \Phi(d)$ then every $\beta_p$ with $p<k$ lies in $\Phi(d)$ as well, $\Phi(d)$ being
an initial section, so that $c_{s_1}(\beta_p)=1$ for all of them. Comparing the
coefficients of $\a_{s_1}$ on the two sides we get
\begin{align*}
1=c_{s_1}(\beta_k)=\sum_{p<k}a_p\,c_{s_1}(\beta_p)+\sum_{q>k}b_q\,c_{s_1}(\beta_q)
\geq\sum_{p<k}a_p\geq2,
\end{align*}
which is absurd. If instead $\beta_k\notin \Phi(d)$, then $c_{s_1}(\beta_k)=0$; since all
the summands lie in $Q^+$, every one of them has $c_{s_1}=0$, that is, every one of them
lies in $\Dp_{(s_1)}$. The displayed identity is therefore a forbidden decomposition for
the order induced on $\Dp_{(s_1)}$, which is excluded by the inductive hypothesis.
\end{proof}

\begin{cor}\label{conj:C1}
A commutation class of reduced decompositions of $w_0$ contains a tangent convex order if and only if it
contains a convex order satisfying {\rm (TR2$'$)}; in that case all its convex orders are tangent.
\end{cor}
\begin{proof}
The ``if'' part is Theorem \ref{thm:tr2ptangent}. Conversely, by Theorem \ref{t43} a tangent convex order is
commutation equivalent to a canonical one, to which Proposition \ref{prop:canonicalTR2p} applies. The last
assertion is Proposition \ref{prop:comminv}.
\end{proof}

\subsection{Projections at tangent roots}\label{projtangent}
The inclusion $\Phi(p_<(\beta_k))\subseteq\beta_k-Q^+$ of Proposition \ref{prop:Bk} holds under a weaker hypothesis than (TR2$'$), namely (TR1) and (TR2) at the root under consideration and at all the roots preceding it, and its conclusion can be sharpened.
\begin{prop}\label{prop:sumfree}
Assume that {\rm (TR1)} and {\rm (TR2)} hold at $\beta_1,\dots,\beta_k$. Then
$\Phi(p_<(\beta_k))$ is abelian and $\Phi(p_<(\beta_k))\subseteq\beta_k-Q^+$.
\end{prop}
\begin{proof}
Put $P^-=\Phi(p_<(\beta_k))\cap(\beta_k-Q^+)$, a subset of $\Phi(w_k)$ containing $\beta_k$,
which is co-closed by the argument given in the proof of Proposition \ref{prop:Bk}. We
show that no two elements of $P^-$ add up to a root. Suppose $\gamma,\delta\in P^-$ and
$\eta=\gamma+\delta\in\Dp$. Then $\eta\in \Phi(p_<(\beta_k))$, the latter being closed, and
convexity places $\eta$ strictly between $\gamma$ and $\delta$; label them so that
$\gamma<\eta<\delta$. Since $\eta\leq\beta_k$, we have $\gamma\neq\beta_k$.
\par By Lemma \ref{lem:proj}(a), applied to $\eta=\gamma+\delta$, we have
$\gamma\in \Phi(p_<(\eta))$, so that Lemma \ref{lem:simpleprec}(b), applied to $p_<(\eta)$,
yields a simple root $\a_r\in \Phi(p_<(\eta))\cap\supp(\gamma)$; and
$\Phi(p_<(\eta))\subseteq \Phi(p_<(\beta_k))$ by Lemma \ref{lem:proj}(b), since
$\eta\in \Phi(p_<(\beta_k))$. As $\beta_k-\gamma\in Q^+\setminus\{0\}$, condition (TR2) at
$\beta_k$ forces $\a_r\notin\supp(\beta_k-\gamma)$, that is $c_r(\gamma)=c_r(\beta_k)$,
where $c_r$ denotes the coefficient of $\a_r$; and $c_r(\gamma)\geq1$. The same argument
gives $c_r(\delta)=c_r(\beta_k)$, trivially so if $\delta=\beta_k$. Hence
$c_r(\eta)=2c_r(\beta_k)\geq2$, so that $\eta\neq\beta_k$ and therefore $\eta<\beta_k$,
and (TR1), (TR2) hold at $\eta$. But $\a_r\in \Phi(p_<(\eta))$, $\a_r<_0\eta$ and
$\a_r\in\supp(\eta-\a_r)$, because $c_r(\eta-\a_r)\geq1$; this contradicts (TR2) at
$\eta$.
\par Thus $P^-$ is closed, and hence biconvex. The minimality of $\Phi(p_<(\beta_k))$ among
the biconvex subsets of $\Phi(w_k)$ containing $\beta_k$ gives
$\Phi(p_<(\beta_k))\subseteq P^-$, so that $\Phi(p_<(\beta_k))=P^-$ and both assertions follow.
\end{proof}

\begin{lem}\label{lem:nosum}
Let $\mu<\beta_k$ with $\supp(\mu)\subseteq\supp(\beta_k)$, and assume that
$\eta=\mu+\beta_k$ is a root. If {\rm (TR1)} holds at $\eta$ then {\rm (TR2)} fails at
$\eta$. In particular, if {\rm (TR1)} and {\rm (TR2)} hold at $\beta_1,\dots,\beta_{k-1}$,
no such $\mu$ exists.
\end{lem}
\begin{proof}
Convexity gives $\mu<\eta<\beta_k$, so that $\mu\in \Phi(p_<(\eta))$ by Lemma \ref{lem:proj}(a). By
Lemma \ref{lem:simpleprec}(b) there is a simple root $\a_c\in \Phi(p_<(\mu))\cap\supp(\mu)$, and
Lemma \ref{lem:proj}(b) at $\eta$ gives $\a_c\in \Phi(p_<(\eta))$. Now $\mu<_0\eta$, with
$\eta-\mu=\beta_k$, and $\a_c\in\supp(\beta_k)\cap \Phi(p_<(\eta))$, contrary to (TR2) at
$\eta$. The last assertion follows because $\eta$ precedes $\beta_k$.
\end{proof}

\vskip5pt
\centerline{\bf\large The second half of the title: quantum tangent spaces}
\vskip5pt
\section{Preliminaries on first order differential calculi and Drinfel'd--Jimbo quantum groups}\label{sec:fodc}
In this section we briefly recall some general definitions and results on covariant first order differential calculi over Hopf algebras. Here and in what follows all algebras are assumed to be unital and defined over $\mathbb{C}$, with the exception of $U_q(\mathfrak{g})$ and its subalgebras, which are defined over $\mathbb{Q}(q)$ (Subsection \ref{subsection:DJQG}). For Hopf algebras we denote by $\Delta$ and $\e$ the coproduct and the counit, respectively.
\subsection{Covariant first order differential calculi and quantum tangent spaces}
Let $B$ be an algebra. A {\em first-order differential calculus (fodc)} over $B$ is a pair $(\Om^1,\exd)$ consisting of a $B$-bimodule $\Omega^1$ and a linear map $\exd: B \to \Omega^1$ such that the {\em Leibniz rule} holds
\begin{align*}
\exd(ab)=a(\exd b)+(\exd a)b,&  & a,b \in B,
\end{align*}
and for which $\Om^1$ is generated as a left $B$-module by those elements of the form~$\exd b$, for~$b \in B$. The {\em universal first-order differential calculus} over $B$ is the pair
$(\Om^1_u(B), \exd_u)$, where $\Om^1_u(B)$ is the kernel of the multiplication $m_B: B \otimes B \to B$ endowed
with the obvious bimodule structure, while the differential $\exd_u$ is defined by
\begin{align*}
\exd_u: B \to \Omega^1_u(B), & & b \mto 1 \otimes b - b \otimes 1.
\end{align*}
Every fodc over $B$ is of the form $\left(\Omega^1_u(B)/N, \,\proj \circ \exd_u\right)$ for some $B$-sub-bimodule $N$ of $\Omega^1_u(B)$, where $\proj:\Omega^1_u(B) \to \Omega^1_u(B)/N$ is
the canonical quotient map. This establishes a bijective correspondence between (isomorphism classes of) first order differential calculi over $B$ and sub-bimodules of $\Omega^1_u(B)$.

Any fodc admits an extension to a differential calculus $\Omega^\bullet(B)$, that is, to a differential graded algebra with $\Omega^0(B)=B$ which is generated in degree $0$ as a differential graded algebra
(see for example \cite[\textsection 1.5]{BeggsMajid:Leabh} for details).

For a Hopf algebra $A$, we say that a fodc $(\Omega^1(A),\exd)$ over $A$ is left covariant if $\Omega^1(A)$ is a left $A$-comodule and $\exd:A\rightarrow\Omega^1(A)$ is a map of left $A$-comodules. Right covariance and bicovariance of a fodc are defined analogously.
If $\Omega^1(A)$ is a left-covariant fodc over $A$, then it is naturally an object of the category ${}^A_A\mathrm{Mod}_A$ and, by the fundamental theorem of Hopf modules,
\begin{align*}
    \Omega^1(A)\cong A \otimes \frac{A^+}{I},
\end{align*}
where $A^+= \ker(\e)$ and $I\subseteq A^+$ is a right ideal of $A$.

Denote by $A^\circ$ the Hopf dual of $A$. A subspace $T\subseteq A^\circ$ is a \textit{left quantum tangent space} for $A$ if
\begin{align*}
    \Delta(T)\subseteq (T\oplus \mathbb{C}) \otimes A^\circ .
\end{align*}
Equivalently, $\mathbb{C}\oplus T$ is a right coideal of $A^\circ$; this is the condition used in Proposition \ref{prop:qtscriterion}. One has an analogous definition of right quantum tangent spaces for right-covariant differential calculi. In the following we will mainly focus on left quantum tangent spaces and, unless otherwise specified, the term quantum tangent space will always mean left quantum tangent space.

There is a one-to-one correspondence between finite-dimensional quantum tangent spaces for $A$ and left-covariant fodc $A\otimes A^+/I$ over $A$ with $\dim A^+/I<\infty$. More precisely, given a quantum tangent space $T$,
\begin{align*}
    I=A^+\cap\ker(T),\qquad \ker(T):=\bigcap_{X \in T} \ker(X),
\end{align*}
is the right ideal of $A$ contained in $A^+$ which defines the fodc corresponding to $T$.
\subsection{Quantized enveloping algebras} \label{subsection:DJQG}
Let $\mathfrak{g}$ be a finite-dimensional complex simple Lie algebra of rank $l$. We fix a Cartan subalgebra $\mathfrak h$ and a set of simple roots $\Pi = \{\alpha_1, \dots, \alpha_l\}$ of the corresponding root system $\D$, and we denote by $\Phi^+$ the corresponding set of positive roots. We use the numbering of the simple roots of \cite{Bo}, and we normalize the invariant form so that $(\alpha,\alpha)=2$ for the short roots $\alpha$. We put $q_i:=q^{\frac{(\alpha_i,\alpha_i)}{2}}$, and we denote by $a_{ij}=\langle\alpha_j,\alpha_i^\vee\rangle$ the entries of the Cartan matrix.
We denote by $U_q(\mathfrak{g})$ the Drinfel'd--Jimbo quantized enveloping algebra, that is, the $\mathbb{Q}(q)$-algebra, $q$ being an indeterminate, generated by $E_i,F_i,K_i,K^{-1}_i$, for $i = 1, \dots, l$, subject to the relations
\begin{gather*}
 K_iE_j = q_i^{a_{ij}} E_j K_i, \quad K_iF_j= q_i^{-a_{ij}} F_j K_i, \quad K_i K_j = K_j K_i, \quad K_iK_i^{-1} = K_i^{-1}K_i = 1,\\
 E_iF_j - F_jE_i = \d_{ij}\frac{K_i - K\inv_{i}}{q_i-q_i^{-1}},
\end{gather*}
along with the quantum Serre relations, which we omit (see for example \cite[\textsection 6.1.2]{KSLeabh} for details). When a numerical value of $q$ is needed, as for the quantum coordinate algebra and the differential calculi below, $q$ is specialized to some $q_0\in\mathbb{C}^\times$ which is not a root of unity; see the beginning of Section \ref{sec:qts}.
We recall that $U_q(\mathfrak{g})$ is a Hopf algebra with coproduct and counit defined on the generators by
\begin{align}\label{coproduct}
   \Delta(K_i)=K_i \otimes K_i, \quad\Delta(E_i)= E_i \otimes K_i + 1 \otimes E_i, \quad \Delta(F_i)= F_i \otimes 1 + K^{-1}_i \otimes F_i,
\end{align}
and
\begin{align*}
   \e(K_i)=1, \quad\e(E_i)= \e(F_i)= 0.
\end{align*}
We denote by $U_q(\mathfrak{n}^+)$ and $U_q(\mathfrak{n}^-)$ the subalgebras of $U_q(\mathfrak{g})$ generated by the $E_i$ and $F_i$ respectively. Similarly we denote by $U_q(\mathfrak{b}^+)$ and $U_q(\mathfrak{b}^-)$ the Hopf subalgebras of $U_q(\mathfrak{g})$ generated respectively by $\{E_i,K^{\pm 1}_i\}$ and $\{F_i,K^{\pm 1}_i\}$. For $\mu=\sum_{i=1}^{l} n_i\alpha_i\in Q$ we set
\begin{align*}
    K_\mu=\prod_{i=1}^{l} K_i^{n_i},
\end{align*}
so that $K_\mu^{-1}=K_{-\mu}$; in particular $K_\beta$ is defined for every $\beta\in\Dp$.
We will denote by $\{\varpi_1, \dots, \varpi_l\}$ the set of fundamental weights of $\g$, by ${\mathcal P}$ the $\bZ$-span of the fundamental weights (the weight lattice) and by ${\mathcal P}^+$ the $\mathbb{Z}_{\geq 0}$-span of the fundamental weights; the root lattice $Q$ and the cone $Q^+$ were introduced in Section \ref{sec:coxeter}. For any $\mu \in \mathcal{P}^+$, we denote by $V_{\mu}$ the associated finite-dimensional irreducible type-$1$ representation of $U_q(\mathfrak{g})$.
We denote by $_{U_q(\mathfrak{g})}\mathbf{type}_1$ the full subcategory of ${U_q(\mathfrak{g})}$-modules whose objects are finite sums of type-$1$ modules $V_\lambda$, for $\lambda \in \mathcal{P}^+$.
\subsection{Quantum coordinate algebra}
Let $V$ be a finite-dimensional $U_q(\mathfrak{g})$-module, $v \in V$, and $f \in V^*$, the linear dual of $V$. Consider the function $c^{\textrm{\tiny $V$}}_{f,v}:U_q(\mathfrak{g}) \to \bC$ defined by $c^{\textrm{\tiny $V$}}_{f,v}(X) := f\big(X(v)\big)$.
The {\em coordinate ring} of $V$ is the subspace
\begin{align*}
C(V) := \text{span}_{\mathbb{C}}\!\left\{ c^{\textrm{\tiny $V$}}_{f,v} \mid v \in V, \, f \in V^*\right\} \sseq U_q(\mathfrak{g})^*.
\end{align*}
It is easily checked that $C(V)$ is contained in $U_q(\mathfrak{g})^\circ$, the Hopf dual of $U_q(\mathfrak{g})$. A Hopf subalgebra of $U_q(\mathfrak{g})^\circ$ is given by
\begin{align*}
\O_q(G) := \bigoplus_{\lambda \in \mathcal{P}^+} C(V_{\lambda}).
\end{align*}
We call $\O_q(G)$ the {\em quantum coordinate algebra of $G$}, where $G$ is the compact, simply-connected, simple Lie group having $\mathfrak{g}$ as its complexified Lie algebra.

\section{Quantum root vectors and reduced expressions}\label{sec:qrv}
In this section we recall some general facts on the quantum root vectors introduced by Lusztig in \cite{Lusztig88,Lusztig90_2} (see also \cite{LusztisLeabh}).
\subsection{\texorpdfstring{Quantum root vectors in $U_q(\mathfrak{g})$}{Quantum root vectors in Uq(g)}}
Let $W$ be the Weyl group of $\mathfrak{g}$. Recall that we have fixed the notation $\Pi=\{\a_1,\ldots,\a_{l}\}$. Unlike in the first part of the paper, we write
$s_i$ for the simple reflection $s_{\a_i}$. It will be convenient to encode reduced words as integer sequences ${\bf i}=(i_1,\ldots,i_k)$, so that saying that ${\bf i}$ is a reduced word for $w$ means that
$w=s_{i_1}\cdots s_{i_k}$.\par
Let moreover $\mathcal B$ be the braid group associated to $W$, i.e. the group whose presentation is obtained from the Coxeter presentation of $W$ with generators $S=\{s_1,\ldots,s_l\}$ by dropping the relations
$s^2_i=e$, $1\leq i\leq l$.
We denote by $T_1, \dots,T_l$ the generators of $\mathcal B$.
We recall the following well-known result (see e.g. \cite[Ch.~8]{Ja}).
\begin{thm}\label{thm:Lusztigquantumrep}
There exists an action of $\mathcal{B}$ on $U_q(\mathfrak{g})$ by algebra automorphisms, given on generators by:
\begin{equation*}
T_i(K_j) = K_{j} K_i^{-a_{ij}}, \ \ T_i(E_i) = -F_{i}K_{i}, \ \ T_i(F_i)= -K_{i}^{-1}E_{i},
\end{equation*}
\begin{equation*}
T_i(E_j) = \sum_{t=0}^{-a_{ij}} (-1)^{t}q_i^{-t}(E_i)^{(-a_{ij}-t)}E_j(E_i)^{(t)}, \ \ i\neq j,
\end{equation*}
\begin{equation*}
T_i(F_j) = \sum_{t=0}^{-a_{ij}} (-1)^{t}q_i^{t}(F_i)^{(t)}F_j(F_i)^{(-a_{ij}-t)}, \ \ i\neq j,
\end{equation*}
where
\begin{equation*}
(E_i)^{(n)} := E_i^n /[n]_{q_i}!, \quad (F_i)^{(n)} := F_i^n / [n]_{q_i}!.
\end{equation*}
\end{thm}
In \cite{Lusztig88} Lusztig used this action to define root vectors in $U_q(\mathfrak{g})$ for every $\beta \in \Phi^+$ as follows.
\begin{defn}\label{defn:qrootvectors}
    Let {\bf i} be a reduced word for $w_0$ and let $\b_j$ be as in \eqref{ordb}. The elements
    \begin{align}
    E^{\bf i}_{\beta_1}= E_{i_1}, \quad  E^{\bf i}_{\beta_k}=T_{i_1}\dots T_{i_{k-1}} ( E_{i_k})
\end{align}
and
\begin{align}\label{eqn:qrootvectorsdefnpos}
    F^{\bf i}_{\beta_1}= F_{i_1}, \quad  F^{\bf i}_{\beta_k}=T_{i_1}\dots T_{i_{k-1}} ( F_{i_k})
\end{align}
for $k \in \{1, \dots, N\},$ where $N=|\Dp|$, are called respectively the positive and negative quantum root vectors corresponding to $\beta_k.$
\end{defn}
In what follows, for $u\in W$ with reduced expression $u=s_{i_1}\dots s_{i_k}$ we denote by $T_{u}\in \mathcal{B}$ the element $T_u=T_{i_1}\dots T_{i_k},$  which is well-defined, i.e., independent of the reduced expression of $u$,  by
Matsumoto's theorem recalled at the beginning of Subsection~\ref{CRFC}.

It is clear from the definition that the total order on $\Phi^+$ induced by {\bf i} gives a total order on the set of positive (resp. negative) quantum root vectors.
The following result is well known; see e.g. \cite[\S 6.2.3, Prop.~23]{KSLeabh}.
\begin{prop}\label{prop:qrootweights}
    All quantum root vectors $E^{\bf i}_{\beta_k}$ and $F^{\bf i}_{\beta_k}$ belong to $U_q(\mathfrak{n}^+)$ and $U_q(\mathfrak{n}^-)$, respectively, and they satisfy
    \begin{align*}
        K_iE^{\bf i}_{\beta_k}K_i^{-1}= q^{(\alpha_i,\beta_k)}E^{\bf i}_{\beta_k}, \quad K_iF^{\bf i}_{\beta_k}K_i^{-1}= q^{-(\alpha_i,\beta_k)}F^{\bf i}_{\beta_k}.
    \end{align*}
\end{prop}
\begin{remark}\label{p74}
    Let $\mathbf{i}$ and $\mathbf{i}'$ be two reduced words for $w_0$ lying in the same
    commutation class. Then $E^{\bf i}_{\beta}=E^{{\bf i}'}_{\beta}$ and $F^{\bf i}_{\beta}=F^{{\bf i}'}_{\beta}$ for every
    $\beta\in\Dp$; in particular the two words determine the same set of quantum
    root vectors, $\{F^{\bf i}_{\beta_k}\}=\{F^{{\bf i}'}_{\beta_k}\}$. This is
    immediate from the facts that $T_iT_j=T_jT_i$, $T_j(E_i)=E_i$ and $T_j(F_i)=F_i$ whenever $(\alpha_i,\alpha_j)=0$; see Theorem \ref{thm:Lusztigquantumrep}, \cite[Ch.~8]{Ja} or
    \cite[Ch.~39]{LusztisLeabh}.
\end{remark}

\begin{remark}\label{rem:projrv}
The projections of Section \ref{sec:projection} give a shortcut for computing quantum root
vectors. Write $p_<(\beta_k)=p'\,s_j$ with $\ell(p')=\ell(p_<(\beta_k))-1$; then
$F^{\bf i}_{\beta_k}=T_{p'}(F_j)$ and $E^{\bf i}_{\beta_k}=T_{p'}(E_j)$.
Indeed $\Phi(p')=\Phi(p_<(\beta_k))\setminus\{\beta_k\}\subseteq \Phi(w_{k-1})$, so that $w_{k-1}=p'u$
with $\ell(w_{k-1})=\ell(p')+\ell(u)$, and $u(\a_{i_k})=p'^{-1}(\beta_k)=\a_j$, whence $T_u(F_{i_k})=F_j$ and
$T_u(E_{i_k})=E_j$ \cite[Ch.~8]{Ja}.
\end{remark}

We end this subsection by recalling the following result on the noncommutative differential geometry of quantum flag manifolds, proved in \cite{ROBPSLusz}, which served as a motivating example for this paper.
\begin{thm}[{\cite{ROBPSLusz}}]
    Let $\mathfrak{g}$ be a simple Lie algebra of type $A_n$. Consider the following two reduced words for $w_0$:
    \begin{align*}
        {\bf i}&=(1,2,\ldots,n,1,2,\ldots,n-1,\ldots,1,2,1),
\\
        {\bf j}&=(n,n-1,\ldots,1,n,n-1,\ldots,2,\ldots, n,n-1,n).
    \end{align*}
    Then each of the sets $\{E^{\bf i}_\beta\mid \b\in \Dp\}$ and $\{E^{\bf j}_\beta\mid \b\in \Dp\}$ spans a quantum tangent space for $\mathcal{O}_q(SU_{n+1})$.
\end{thm}
\begin{remark} Both {\bf i} and {\bf j} are tangent reduced decompositions in the sense of Subsection \ref{tco}: they induce canonical tangent convex orders. Hence, by Theorem \ref{thm:combinatorialqts}, the elements $F^{\bf i}_\beta K_\beta$ (resp. $F^{\bf j}_\beta K_\beta$), $\beta\in\Dp$, span quantum tangent spaces as well. Note that \cite{ROBPSLusz} works with the positive quantum root vectors $E^{\bf i}_\beta$, whereas the quantum tangent spaces considered in this paper are spanned by the negative ones multiplied by $K_\beta$.\end{remark}
In the same article the authors showed that the corresponding calculi induce two fodc on the full quantum flag manifold, which are suitable $q$-deformations of the holomorphic calculus on the classical full flag manifold of type $A_n$.

\section{The coproduct of quantum root vectors and convex orders}\label{sec:coproduct}
In this section we describe the coproduct of the quantum root vectors
$\{F^{\bf i}_{\beta}\}$ associated to the choice of a reduced decomposition
${\bf i}$ of $w_0$ in terms of the convex order induced by ${\bf i}$ itself:
the coproduct separates the root vectors that precede a given
$F^{\bf i}_{\beta_k}$ from the ones that follow it, by placing them in the left and in the
right tensor factor, respectively.
\par The result is in substance known, being equivalent to the assertion that, for every $k$,
the quantum Schubert cell attached to $w_k=s_{i_1}\cdots s_{i_k}$, multiplied by the subalgebra $U^0$ generated by the $K_i^{\pm1}$, is a right
coideal subalgebra of $U_q(\mathfrak{g})$, while the subalgebra generated by the root vectors attached to $\beta_{k},\dots,\beta_N$ is a left
coideal subalgebra; see \cite{LS,DCKP} and, for
case-free proofs of far-reaching generalizations in the setting of Nichols
algebras, \cite{HS}. We give a short proof based on the Drinfeld pairing
between $U_{q}(\mathfrak{b}^+)$ and $U_{q}(\mathfrak{b}^-)$. The pairing exchanges the coproduct of $\Uminus$
with the product of $\Uplus$, so that the coefficients of
$\Delta F^{\bf i}_{\beta_k}$ \emph{are}, up to nonzero factors, structure constants of the
multiplication of PBW monomials in $\Uplus$; the required triangularity then
follows from the fact that the subalgebras generated by consecutive quantum
root vectors possess their own PBW bases. The argument is uniform in the type---in particular, type $G_2$ needs no separate treatment---and it yields as a
by-product the explicit non-vanishing criterion of
Corollary~\ref{cor:criterion}, which is what is actually used in Section \ref{sec:qts}.

\subsection{Notation}\label{sec:pbwnot}
Throughout this section ${\bf i}=(i_1,\dots,i_N)$ is a fixed reduced word for
$w_0$, $N=|\Dp|$, and $\beta_k=\beta^{\bf i}_k$, so that
$\beta_1<_{\bf i}\cdots<_{\bf i}\beta_N$ is the convex order induced by
${\bf i}$. We write $E_{\beta_k}, F_{\beta_k}$ for the quantum root vectors of
Definition \ref{defn:qrootvectors}, dropping the superscript ${\bf i}$, which is fixed once and for
all. For $\bfa=(a_1,\dots,a_N)\in\bN^{N}$ we put
\begin{align*}
E^{\bfa}:=E_{\beta_1}^{a_1}\cdots E_{\beta_N}^{a_N},\qquad
F^{\bfa}:=F_{\beta_1}^{a_1}\cdots F_{\beta_N}^{a_N},
\end{align*}
\begin{align*}
|\bfa|:=\sum_{j=1}^{N}a_j\beta_j\in Q^+,\qquad
\supp(\bfa):=\{j\mid a_j\neq0\}\subseteq[1,N],
\end{align*}
and we let $\bfe_k\in\bN^N$ be the $k$-th coordinate vector, so that
$F^{\bfe_k}=F_{\beta_k}$ and $|\bfe_k|=\beta_k$. For $J\subseteq[1,N]$ we
denote by $\bfa|_J$ the vector obtained from $\bfa$ by setting $a_j=0$ for
$j\notin J$. We freely use the PBW theorem: $\{E^{\bfa}\}$ and $\{F^{\bfa}\}$,
$\bfa\in\bN^N$, are $\Qq$-bases of $\Uplus$ and $\Uminus$, respectively
\cite{LusztisLeabh,DCKP}. Finally, for $\mu\in Q^+$ we let
$\Uplus_{\mu}$ and $\Uminus_{-\mu}$ be the weight spaces, so that
$E^{\bfa}\in\Uplus_{|\bfa|}$ and $F^{\bfa}\in\Uminus_{-|\bfa|}$ by
Proposition \ref{prop:qrootweights}.

\subsection{The shape of the coproduct}
The following elementary lemma is what makes the coefficients of Theorem
\ref{copro} well defined.

\begin{lem}\label{lem:shape}
For every $\mu\in Q^+$,
\begin{align*}
\Delta\bigl(\Uminus_{-\mu}\bigr)&\subseteq
\bigoplus_{\mu_1+\mu_2=\mu}\Uminus_{-\mu_1}K^{-1}_{\mu_2}\otimes \Uminus_{-\mu_2},
\\
\Delta\bigl(\Uplus_{\mu}\bigr)&\subseteq
\bigoplus_{\mu_1+\mu_2=\mu}\Uplus_{\mu_1}\otimes \Uplus_{\mu_2}K_{\mu_1},
\end{align*}
the sums being over $\mu_1,\mu_2\in Q^+$.
\end{lem}
\begin{proof}
Both inclusions hold for the generators by \eqref{coproduct}, and they are preserved by products, since
$K_\lambda y=q^{-(\lambda,\nu)}yK_\lambda$ for $y\in \Uminus_{-\nu}$ and
$K_\lambda x=q^{(\lambda,\nu)}xK_\lambda$ for $x\in \Uplus_{\nu}$.
\end{proof}

By Lemma \ref{lem:shape} and the PBW theorem there are uniquely determined
$c_{\bfa,\bfb},\ \tilde c_{\bfa,\bfb}\in\Qq$ such that
\begin{align}
\Delta F_{\beta_k}&=\sum_{\substack{\bfa,\bfb\in\bN^N\\ |\bfa|+|\bfb|=\beta_k}}
c_{\bfa,\bfb}\;F^{\bfa}K^{-1}_{|\bfb|}\otimes F^{\bfb}, \label{eqn:defc}
\\
\Delta E_{\beta_k}&=\sum_{\substack{\bfa,\bfb\in\bN^N\\ |\bfa|+|\bfb|=\beta_k}}
\tilde c_{\bfa,\bfb}\;E^{\bfa}\otimes E^{\bfb}K_{|\bfa|}. \label{eqn:defct}
\end{align}
In this notation, Theorem \ref{copro} is precisely the assertion that
$c_{\bfa,\bfb}=0$ unless $\supp(\bfa)\subseteq[1,k-1]$ and
$\supp(\bfb)\subseteq[k+1,N]$, apart from the two extreme terms.

\begin{lem}\label{lem:extreme}
$c_{\bfa,\bfze}=\delta_{\bfa,\bfe_k}$ and
$c_{\bfze,\bfb}=\delta_{\bfb,\bfe_k}$. Consequently the terms of
\eqref{eqn:defc} with $\bfa=\bfze$ or $\bfb=\bfze$ are exactly
$K^{-1}_{\beta_k}\otimes F_{\beta_k}$ and $F_{\beta_k}\otimes 1$.
\end{lem}
\begin{proof}
Apply $\id\otimes\e$ to \eqref{eqn:defc}. Since
$\e(F^{\bfb})=\delta_{\bfb,\bfze}$ and $\e(K_\mu)=1$, the
counit axiom gives $\sum_{\bfa}c_{\bfa,\bfze}F^{\bfa}=F_{\beta_k}$, whence
$c_{\bfa,\bfze}=\delta_{\bfa,\bfe_k}$ by the PBW theorem. Applying
$\e\otimes\id$ gives $\sum_{\bfb}c_{\bfze,\bfb}F^{\bfb}=F_{\beta_k}$
in the same way.
\end{proof}

\begin{lem}\label{lem:nok}
Let $\bfa,\bfb\in\bN^N\setminus\{\bfze\}$ with $|\bfa|+|\bfb|=\beta_k$. Then
$k\notin\supp(\bfa)\cup\supp(\bfb)$.
\end{lem}
\begin{proof}
If $a_k\geq1$ then $|\bfa|-\beta_k\in Q^+$, so that
$0=(|\bfa|-\beta_k)+|\bfb|$ is a sum of two elements of $Q^+$, forcing
$|\bfb|=0$, i.e.\ $\bfb=\bfze$, a contradiction. The same argument applies to
$\bfb$.
\end{proof}

Thus, for a non-extreme term of \eqref{eqn:defc}, the condition
``$\supp(\bfa)\subseteq[1,k-1]$ and $\supp(\bfb)\subseteq[k+1,N]$'' is
equivalent to ``$\supp(\bfa)\cap[k+1,N]=\emptyset$ and
$\supp(\bfb)\cap[1,k-1]=\emptyset$''.

\subsection{The Drinfeld pairing}
We recall the pairing between the two Borel parts of $U_q(\mathfrak{g})$; see
\cite[Ch.~6]{Ja}, \cite[\S 6.3]{KSLeabh},
\cite[Ch.~1 and \S 3.1]{LusztisLeabh}, \cite{Ta}. Since several
normalizations coexist in the literature, we state the properties we use in the
form adapted to the coproduct \eqref{coproduct}.

\begin{prop}\label{prop:pairing}
There is a bilinear form
$\pair{\cdot}{\cdot}\colon U_{q}(\mathfrak{b}^+)\times U_{q}(\mathfrak{b}^-)\to\Qq$ such that, for all
$x,x'\in U_{q}(\mathfrak{b}^+)$, $y,y'\in U_{q}(\mathfrak{b}^-)$, $\mu,\nu\in Q$ and $i,j$:
\begin{enumerate}
\item[\rm(D1)] $\pair{x}{yy'}=\pair{\Delta(x)}{y\otimes y'}$;
\item[\rm(D2)] $\pair{xx'}{y}=\pair{x\otimes x'}{\Delta^{\mathrm{op}}(y)}$;
\item[\rm(D3)] $\pair{K_\mu}{K_\nu}=q^{(\mu,\nu)}$,
$\pair{K_\mu}{F_j}=0$, $\pair{E_i}{K_\mu}=0$;
\item[\rm(D4)] $\pair{E_i}{F_j}=\delta_{ij}d_i$ with $d_i\in\Qq^{\times}$
(e.g.\ $d_i=(q_i^{-1}-q_i)^{-1}$).
\end{enumerate}
Moreover $\pair{\Uplus_{\mu}K_{\lambda}}{\Uminus_{-\nu}K_{\eta}}=0$ whenever
$\mu\neq\nu$, and the restriction
\[\pair{\cdot}{\cdot}\colon \Uplus_{\mu}\times\Uminus_{-\mu}\to\Qq\] is
nondegenerate for every $\mu\in Q^+$.
\end{prop}

\begin{remark}
Some references state (D1)--(D2) with $\Delta$ and $\Delta^{\mathrm{op}}$
interchanged, together with $\pair{K_\mu}{K_\nu}=q^{-(\mu,\nu)}$; the two sets of conventions differ by composing the second argument with the inverse of the antipode. Given the coproduct \eqref{coproduct}, the value of $\pair{K_\mu}{K_\nu}$ in (D3) is forced by (D2) and by the relation $K_\mu E_i=q^{(\mu,\alpha_i)}E_iK_\mu$: by (D2), $\pair{K_\mu E_i}{F_i}=d_i$, while $\pair{E_iK_\mu}{F_i}=d_i\pair{K_\mu}{K_{-\alpha_i}}$.
Nothing below depends on the value of the constants $d_i$.
\end{remark}

\begin{lem}\label{lem:Ktwist}
Let $x\in \Uplus_{\mu}$, $y\in \Uminus_{-\mu}$ and $\lambda,\eta\in Q$. Then
\begin{align*}
\pair{x}{yK_{\eta}}=q^{(\mu,\eta)}\pair{x}{y},\qquad
\pair{xK_{\lambda}}{y}=q^{-(\lambda,\mu)}\pair{x}{y}.
\end{align*}
\end{lem}
\begin{proof}
By (D1), $\pair{x}{yK_\eta}=\pair{\Delta x}{y\otimes K_\eta}$. By Lemma
\ref{lem:shape} we may write
\[\Delta x=\sum_{\mu_1+\mu_2=\mu}x_{\mu_1}\otimes x'_{\mu_2}K_{\mu_1},\,
x_{\mu_1}\in \Uplus_{\mu_1}, x'_{\mu_2}\in \Uplus_{\mu_2}.\] By the weight
orthogonality in Proposition \ref{prop:pairing}, the second tensor slot pairs
nontrivially with $K_\eta$ only if $\mu_2=0$, and the corresponding summand is
$x\otimes K_{\mu}$. Hence
$\pair{x}{yK_\eta}=\pair{x}{y}\pair{K_\mu}{K_\eta}=q^{(\mu,\eta)}\pair{x}{y}$.
The second identity is proved in the same way using (D2) and the first
inclusion of Lemma \ref{lem:shape}: the surviving summand of
$\Delta^{\mathrm{op}}(y)$ is $y\otimes K^{-1}_{\mu}$, and
$\pair{K_\lambda}{K_{-\mu}}=q^{-(\lambda,\mu)}$.
\end{proof}

The second ingredient is the orthogonality of the PBW bases. It is the reason
why the whole argument works, and it is exactly the point at which the choice
of the convex order $<_{\bf i}$ enters.

\begin{prop}\label{prop:orth}
There are constants $\kappa_{\bfa}\in\Qq^{\times}$, $\bfa\in\bN^N$, such that
\begin{align*}
\pair{E^{\bfa}}{F^{\bfb}}=\delta_{\bfa,\bfb}\,\kappa_{\bfa},\qquad
\bfa,\bfb\in\bN^{N},
\end{align*}
where both monomials are taken in the order induced by the \emph{same} reduced
word ${\bf i}$.
\end{prop}
This is classical: see \cite[Prop.~38.2.3]{LusztisLeabh} for Lusztig's form on
$\mathbf f$, which is related to the restriction of $\pair{\cdot}{\cdot}$ to
$\Uplus_{\mu}\times\Uminus_{-\mu}$ by a normalization and the substitution $v=q^{-1}$, as well as \cite{LS}
and \cite[\S 4]{Ki}; a treatment adapted to arbitrary elements of the braid
monoid is given in \cite{Ya}. Only $\kappa_{\bfa}\neq0$ is used below.

\subsection{Interval subalgebras}
Let $X\in\{E,F\}$ and write $X_j:=X_{\beta_j}$.

\begin{lem}\label{lem:interval}
For $1\leq m\leq n\leq N$ let $A^{X}[m,n]\subseteq U_q(\mathfrak{g})$ denote
the subalgebra generated by $X_{\beta_m},\dots,X_{\beta_n}$. Then
\begin{align*}
\{\,X^{\bfc}\ \mid\ \supp(\bfc)\subseteq[m,n]\,\}
\end{align*}
is a $\Qq$-basis of $A^{X}[m,n]$. Moreover, for every $1\leq m<N$
multiplication induces a $\Qq$-linear isomorphism
\begin{align*}
A^{X}[1,m]\otimes A^{X}[m+1,N]\ \xrightarrow{\ \sim\ }\ A^{X}[1,N],
\end{align*}
and every PBW monomial factors accordingly as
$X^{\bfc}=X^{\bfc|_{[1,m]}}\,X^{\bfc|_{[m+1,N]}}$.
\end{lem}
\begin{proof}
Linear independence is clear, the displayed set being a subset of the PBW basis
of $U_q(\mathfrak{n}^{\pm})$. For the spanning property, set
$w_{m-1}=s_{i_1}\cdots s_{i_{m-1}}$ and $v=s_{i_m}\cdots s_{i_n}$, so that
$(i_m,\dots,i_n)$ is a reduced word for $v$; denote by
$X^{v}_{\gamma_1},\dots,X^{v}_{\gamma_{n-m+1}}$ the quantum root vectors
attached to it. Definition \ref{defn:qrootvectors} gives, for $1\leq r\leq n-m+1$,
\begin{align*}
X_{\beta_{m-1+r}}
=T_{i_1}\cdots T_{i_{m-1}}\bigl(T_{i_m}\cdots T_{i_{m+r-2}}(X_{i_{m+r-1}})\bigr)
=T_{w_{m-1}}\bigl(X^{v}_{\gamma_r}\bigr),
\end{align*}
$T_{w_{m-1}}$ being well defined by Matsumoto's theorem. Since $T_{w_{m-1}}$ is
an algebra automorphism of $U_q(\mathfrak{g})$ by Theorem
\ref{thm:Lusztigquantumrep}, we get
$A^{X}[m,n]=T_{w_{m-1}}\bigl(U^{\pm}[v]\bigr)$, where $U^{\pm}[v]$ is the
quantum Schubert cell of \cite{DCKP,LusztisLeabh}. By \cite[Prop.~2.2]{DCKP}
(see also \cite[Prop.~40.2.1]{LusztisLeabh}) the ordered monomials in the
$X^{v}_{\gamma_r}$ form a basis of $U^{\pm}[v]$; applying the algebra
automorphism $T_{w_{m-1}}$ and using the displayed identity gives the claim.
The last two assertions are immediate from the PBW theorem.
\end{proof}

We can now prove the combinatorial heart of the matter.

\begin{prop}\label{prop:key}
Let $\bfu,\bfv\in\bN^{N}$ and write
$X^{\bfu}X^{\bfv}=\sum_{\bfc}\lambda_{\bfc}X^{\bfc}$ in the PBW basis. Fix
$k\in[1,N]$.
\begin{enumerate}
\item If $\supp(\bfu)\cap[1,k-1]\neq\emptyset$, then every $\bfc$ with
$\lambda_{\bfc}\neq0$ satisfies $\supp(\bfc)\cap[1,k-1]\neq\emptyset$.
\item If $\supp(\bfv)\cap[k+1,N]\neq\emptyset$, then every $\bfc$ with
$\lambda_{\bfc}\neq0$ satisfies $\supp(\bfc)\cap[k+1,N]\neq\emptyset$.
\end{enumerate}
In particular, if $\supp(\bfu)\cap[1,k-1]\neq\emptyset$ or
$\supp(\bfv)\cap[k+1,N]\neq\emptyset$, then $\lambda_{\bfe_k}=0$, that is,
$X_{\beta_k}$ does not occur in $X^{\bfu}X^{\bfv}$.
\end{prop}
\begin{proof}
(1) Write $\bfu=\bfu'+\bfu''$ with $\bfu'=\bfu|_{[1,k-1]}\neq\bfze$ and
$\bfu''=\bfu|_{[k,N]}$, so that $X^{\bfu}=X^{\bfu'}X^{\bfu''}$ and
\begin{align*}
X^{\bfu}X^{\bfv}=X^{\bfu'}\bigl(X^{\bfu''}X^{\bfv}\bigr).
\end{align*}
Expand $X^{\bfu''}X^{\bfv}=\sum_{\bfc}\lambda'_{\bfc}X^{\bfc}$ and split
$X^{\bfc}=X^{\bfc'}X^{\bfc''}$, where $\bfc'=\bfc|_{[1,k-1]}$ and
$\bfc''=\bfc|_{[k,N]}$, by Lemma \ref{lem:interval}. Then
$X^{\bfu'}X^{\bfc}=\bigl(X^{\bfu'}X^{\bfc'}\bigr)X^{\bfc''}$ and
$X^{\bfu'}X^{\bfc'}\in A^{X}[1,k-1]$, so that, again by Lemma
\ref{lem:interval},
$X^{\bfu'}X^{\bfc'}=\sum_{\bfd}\mu_{\bfd}X^{\bfd}$ with
$\supp(\bfd)\subseteq[1,k-1]$; moreover $\bfd\neq\bfze$ whenever
$\mu_{\bfd}\neq0$, because $|\bfd|=|\bfu'|+|\bfc'|\neq0$. Finally
$X^{\bfd}X^{\bfc''}=X^{\bfd+\bfc''}$ is already a PBW monomial, since
$\supp(\bfd)\subseteq[1,k-1]$ and $\supp(\bfc'')\subseteq[k,N]$. Hence every
PBW monomial occurring in $X^{\bfu}X^{\bfv}$ has support meeting $[1,k-1]$.
\par (2) Symmetrically, write $\bfv=\bfv'+\bfv''$ with $\bfv'=\bfv|_{[1,k]}$
and $\bfv''=\bfv|_{[k+1,N]}\neq\bfze$, so that
$X^{\bfu}X^{\bfv}=\bigl(X^{\bfu}X^{\bfv'}\bigr)X^{\bfv''}$. Expanding
$X^{\bfu}X^{\bfv'}=\sum_{\bfc}\lambda'_{\bfc}X^{\bfc}$, splitting
$X^{\bfc}=X^{\bfc'}X^{\bfc''}$ with $\bfc'=\bfc|_{[1,k]}$,
$\bfc''=\bfc|_{[k+1,N]}$, and using that
$X^{\bfc''}X^{\bfv''}\in A^{X}[k+1,N]$ expands into monomials $X^{\bfd}$ with
$\bfd\neq\bfze$ and $\supp(\bfd)\subseteq[k+1,N]$, we conclude as before that
$X^{\bfc'}X^{\bfd}=X^{\bfc'+\bfd}$ is a PBW monomial whose support meets
$[k+1,N]$.
\par The last assertion follows because $\supp(\bfe_k)=\{k\}$ meets neither
$[1,k-1]$ nor $[k+1,N]$.
\end{proof}

\subsection{The coproduct of quantum root vectors}

\begin{lem}\label{lem:dual}
With the notation of \eqref{eqn:defc} and \eqref{eqn:defct}, for all
$\bfa,\bfb\in\bN^{N}$ with $|\bfa|+|\bfb|=\beta_k$ one has
\begin{align}
\pair{E^{\bfb}E^{\bfa}}{F_{\beta_k}}
&=q^{-(|\bfa|,|\bfb|)}\,\kappa_{\bfa}\kappa_{\bfb}\,c_{\bfa,\bfb},
\label{eqn:dualF}
\\
\pair{E_{\beta_k}}{F^{\bfa}F^{\bfb}}
&=q^{-(|\bfa|,|\bfb|)}\,\kappa_{\bfa}\kappa_{\bfb}\,\tilde c_{\bfa,\bfb}.
\label{eqn:dualE}
\end{align}
\end{lem}
\begin{proof}
For \eqref{eqn:dualF}, apply (D2) to $x=E^{\bfb}$, $x'=E^{\bfa}$ and
$y=F_{\beta_k}$:
\begin{align*}
\pair{E^{\bfb}E^{\bfa}}{F_{\beta_k}}
=\pair{E^{\bfb}\otimes E^{\bfa}}{\Delta^{\mathrm{op}}F_{\beta_k}}
=\sum_{\bfa',\bfb'}c_{\bfa',\bfb'}\,
\pair{E^{\bfb}}{F^{\bfb'}}\;\pair{E^{\bfa}}{F^{\bfa'}K^{-1}_{|\bfb'|}}.
\end{align*}
By Proposition \ref{prop:orth} the first factor is
$\delta_{\bfb,\bfb'}\kappa_{\bfb}$; by Lemma \ref{lem:Ktwist} and Proposition
\ref{prop:orth} the second one equals
$q^{-(|\bfa'|,|\bfb'|)}\delta_{\bfa,\bfa'}\kappa_{\bfa}$. Only the summand
$(\bfa',\bfb')=(\bfa,\bfb)$ survives, which is \eqref{eqn:dualF}. The proof of
\eqref{eqn:dualE} is identical, using (D1) and the second identity of Lemma
\ref{lem:Ktwist}.
\end{proof}

\begin{lem}\label{lem:transfer}
Write $E^{\bfb}E^{\bfa}=\sum_{\bfc}\lambda_{\bfc}E^{\bfc}$ and
$F^{\bfa}F^{\bfb}=\sum_{\bfc}\tilde\lambda_{\bfc}F^{\bfc}$. Then
\begin{align*}
\pair{E^{\bfb}E^{\bfa}}{F_{\beta_k}}=\kappa_{\bfe_k}\lambda_{\bfe_k},\qquad
\pair{E_{\beta_k}}{F^{\bfa}F^{\bfb}}=\kappa_{\bfe_k}\tilde\lambda_{\bfe_k}.
\end{align*}
\end{lem}
\begin{proof}
Immediate from Proposition \ref{prop:orth} and bilinearity.
\end{proof}

\begin{thm}\label{copro}
The coproduct of $F^{\bf i}_{\beta_k}$ is given by
\begin{align}\label{eqn:coproductqroot}
&\Delta F^{\bf i}_{\beta_k} =
\\
&F^{\bf i}_{\beta_k} \otimes 1 +
\sum_{\substack{p_1\leq\dots \leq p_r< k < q_1\leq \dots \leq q_s,\ r,s\geq1 \\
\beta_{p_1}+\dots + \beta_{p_r}+\beta_{q_1}+ \dots + \beta_{q_s}=\beta_k}}
c_{p_1,\dots ,p_r}^{q_1, \dots ,q_s}
F^{\bf i}_{\beta_{p_1}}\dots F^{\bf i}_{\beta_{p_r}} K^{-1}_{\mathbf q}\otimes
F^{\bf i}_{\beta_{q_1}}\dots F^{\bf i}_{\beta_{q_s}} \nonumber
\\
&\qquad+ K_{\beta_k}^{-1}\otimes F^{\bf i}_{\beta_k}, \nonumber
\end{align}
where $K^{-1}_{\mathbf q}:=K^{-1}_{\beta_{q_1}}\dots K^{-1}_{\beta_{q_s}}$ and
$c_{p_1,\dots ,p_r}^{q_1, \dots ,q_s}\in\Qq$. These coefficients are in fact
Laurent polynomials in $q$, the coproduct preserving the integral form of
\cite[\S 1]{DCKP} in which the PBW monomials form a $\mathbb{Z}[q,q^{-1}]$-basis.
Equivalently, in the notation of \S \ref{sec:pbwnot},
\begin{align*}
\Delta F_{\beta_k}
= F_{\beta_k}\otimes 1
+\sum_{\substack{\bfa,\bfb\in\bN^{N}\setminus\{\bfze\}\\
\supp(\bfa)\subseteq[1,k-1],\ \supp(\bfb)\subseteq[k+1,N]\\
|\bfa|+|\bfb|=\beta_k}}
c_{\bfa,\bfb}\;F^{\bfa}K^{-1}_{|\bfb|}\otimes F^{\bfb}
+K^{-1}_{\beta_k}\otimes F_{\beta_k}.
\end{align*}
\end{thm}
\begin{proof}
By Lemma \ref{lem:shape} the expansion \eqref{eqn:defc} exists and is unique,
and by Lemma \ref{lem:extreme} its terms with $\bfa=\bfze$ or $\bfb=\bfze$ are
exactly the two extreme terms displayed above. So let $\bfa,\bfb\neq\bfze$ with
$|\bfa|+|\bfb|=\beta_k$ and assume $c_{\bfa,\bfb}\neq0$. Writing
$E^{\bfb}E^{\bfa}=\sum_{\bfc}\lambda_{\bfc}E^{\bfc}$, Lemmas \ref{lem:dual} and
\ref{lem:transfer} give
\begin{align*}
\kappa_{\bfe_k}\lambda_{\bfe_k}
=q^{-(|\bfa|,|\bfb|)}\kappa_{\bfa}\kappa_{\bfb}c_{\bfa,\bfb}\neq0,
\end{align*}
so that $E_{\beta_k}$ occurs in the PBW expansion of $E^{\bfb}E^{\bfa}$.
Applying Proposition \ref{prop:key} with $X=E$, $\bfu=\bfb$, $\bfv=\bfa$ we
conclude that
\begin{align*}
\supp(\bfb)\cap[1,k-1]=\emptyset\qquad\text{and}\qquad
\supp(\bfa)\cap[k+1,N]=\emptyset .
\end{align*}
By Lemma \ref{lem:nok} moreover $k\notin\supp(\bfa)\cup\supp(\bfb)$, whence
$\supp(\bfa)\subseteq[1,k-1]$ and $\supp(\bfb)\subseteq[k+1,N]$, as claimed.
\end{proof}

\subsection{A criterion for the non-vanishing of the coefficients}
The identity \eqref{eqn:dualF} combined with Lemma \ref{lem:transfer} gives the following result, which is what is
actually needed in Section \ref{sec:qts}.

\begin{cor}\label{cor:criterion}
Let $\bfa,\bfb\in\bN^{N}\setminus\{\bfze\}$ with $|\bfa|+|\bfb|=\beta_k$. Then
\begin{align*}
c_{\bfa,\bfb}\neq0
\iff
E_{\beta_k}\ \text{occurs in the PBW expansion of }\ E^{\bfb}E^{\bfa}.
\end{align*}
\end{cor}

\section{\texorpdfstring{Reduced decompositions of $w_0$ and quantum tangent spaces}{Reduced decompositions of w0 and quantum tangent spaces}}\label{sec:qts}
In the previous section we have established the role of convex orders in the coproduct of quantum root vectors.
In this section we prove that the reduced decompositions that give rise to a quantum tangent space are precisely the tangent reduced decompositions introduced in Subsection \ref{tco}.
Throughout, as in Section \ref{sec:coproduct}, $q$ is an indeterminate and
$U_q(\mathfrak{g})$ is taken over $\Qq$, so that the coefficients $c_{\bfa,\bfb}$ are
Laurent polynomials in $q$. The results below remain valid after specializing $q$ to a fixed
$q_0\in\mathbb{C}^\times$ which is not a root of unity, provided $q_0$ is not a zero of the finitely many nonzero coefficients involved in the argument; this holds for all but finitely many
such $q_0$ and, these coefficients being Laurent polynomials with rational coefficients, for every transcendental
$q_0$.
\subsection{The combinatorial properties of a reduced decomposition associated to a quantum tangent space}
The next proposition reduces the quantum tangent space condition to a question about
PBW monomials: by Corollary \ref{cor:criterion}, it becomes a statement about the
multiplication of PBW monomials in $\Uplus$.

\begin{prop}\label{prop:qtscriterion}
The set $\{F^{\bf i}_{\beta_k}K_{\beta_k}\mid 1\leq k \leq N\}$ spans a quantum
tangent space if and only if, for every $1\leq k\leq N$ and every pair
$\bfa,\bfb\in\bN^{N}\setminus\{\bfze\}$ with $|\bfa|+|\bfb|=\beta_k$ and
$\sum_{j}a_j\geq2$, the quantum root vector $E_{\beta_k}$ does \emph{not} occur in
the PBW expansion of $E^{\bfb}E^{\bfa}$.
\end{prop}
\begin{proof}
Put $T=\mathrm{span}_{\Qq}\{F_{\beta_m}K_{\beta_m}\}$. Multiplying
\eqref{eqn:coproductqroot} by $\Delta(K_{\beta_k})=K_{\beta_k}\otimes K_{\beta_k}$, and
using $K^{-1}_{|\bfb|}K_{\beta_k}=K_{\beta_k-|\bfb|}=K_{|\bfa|}$ on the middle terms
together with $K^{-1}_{\beta_k}K_{\beta_k}=1$ on the last one, we obtain
\begin{align}\label{eqn:coproKbeta}
\Delta\bigl(F^{\bf i}_{\beta_k}K_{\beta_k}\bigr)
= F^{\bf i}_{\beta_k}K_{\beta_k}\otimes K_{\beta_k}
+\sum_{\bfa,\bfb}c_{\bfa,\bfb}\,F^{\bfa}K_{|\bfa|}\otimes F^{\bfb}K_{\beta_k}
+1\otimes F^{\bf i}_{\beta_k}K_{\beta_k},
\end{align}
the sum being over the pairs $\bfa,\bfb\in\bN^N\setminus\{\bfze\}$ of Theorem
\ref{copro}. By definition $T$ is a quantum tangent space if and only if
$\Delta(T)\subseteq(T\oplus\Qq)\otimes U_q(\mathfrak{g})$, that is, if and only if every
left tensor factor of \eqref{eqn:coproKbeta} lies in $T\oplus\Qq$. The
two extreme summands contribute $F_{\beta_k}K_{\beta_k}\in T$ and $1\in\Qq$, so
only the middle sum matters.
\par The elements $F^{\bfb}K_{\beta_k}$, $\bfb\in\bN^N$, are linearly
independent, $\{F^{\bfb}\}$ being a PBW basis and $K_{\beta_k}$ being invertible;
hence the condition has to hold separately for each $\bfb$. Fix such a $\bfb$. All
the $\bfa$ paired with it satisfy $|\bfa|=\beta_k-|\bfb|=:\mu$, so that they share
the same factor $K_{\mu}$ and the corresponding left tensor factor is
\begin{align*}
\Bigl(\sum_{|\bfa|=\mu}c_{\bfa,\bfb}F^{\bfa}\Bigr)K_{\mu},
\qquad \mu\in Q^+\setminus\{0\}.
\end{align*}
Now $\bigl(T\oplus\Qq\bigr)\cap \Uminus_{-\mu}K_{\mu}$ equals
$\Qq\,F_{\beta_p}K_{\beta_p}$ if $\mu=\beta_p\in\Dp$, and $0$ otherwise: a nonzero
scalar has $\Uminus$-weight $0\neq-\mu$, while
$F_{\beta_m}K_{\beta_m}\in \Uminus_{-\beta_m}K_{\beta_m}$ forces $\beta_m=\mu$.
Since the $F^{\bfa}$ with $|\bfa|=\mu$ are linearly independent, and since
$\bfa=\bfe_p$ is the unique such $\bfa$ with $\sum_j a_j=1$, the displayed element
lies in $T\oplus\Qq$ if and only if $c_{\bfa,\bfb}=0$ for every $\bfa$ with
$\sum_j a_j\geq2$; no cancellation between different $\bfa$ is possible.
Corollary \ref{cor:criterion} now translates $c_{\bfa,\bfb}=0$ into the
non-occurrence of $E_{\beta_k}$ in $E^{\bfb}E^{\bfa}$.
\end{proof}

By Theorem \ref{copro}, the criterion of Proposition \ref{prop:qtscriterion} is satisfied as
soon as no $\beta_k$ admits a decomposition of the shape \eqref{eqn:baddec}, that is, by Proposition \ref{prop:nobaddec}, as soon as (TR2$'$) holds.

\begin{cor}\label{cor:tangentimpliesqts}
If the convex order induced by ${\bf i}$ satisfies {\rm (TR2$'$)}, then
$\{F^{\bf i}_{\beta_k}K_{\beta_k}\mid 1\leq k\leq N\}$ spans a quantum tangent
space.
\end{cor}
\begin{proof}
Let $\bfa,\bfb\in\bN^N\setminus\{\bfze\}$ with $|\bfa|+|\bfb|=\beta_k$ and
$\sum_j a_j\geq2$. If $\supp(\bfa)\not\subseteq[1,k-1]$ or
$\supp(\bfb)\not\subseteq[k+1,N]$, then $c_{\bfa,\bfb}=0$ by Theorem \ref{copro},
so $E_{\beta_k}$ does not occur in $E^{\bfb}E^{\bfa}$ by Corollary
\ref{cor:criterion}. If instead both inclusions hold, then $|\bfa|+|\bfb|=\beta_k$
is an identity of the shape \eqref{eqn:baddec}, which Proposition
\ref{prop:nobaddec} excludes. Hence the condition of Proposition \ref{prop:qtscriterion} holds for every such pair.
\end{proof}

\subsection{Reduction to a statement on root systems}
\label{sec:reduction}
Throughout this subsection we write
\begin{align*}
\mathcal{T}_{\bf i}:=\mathrm{span}\{F^{\bf i}_{\beta_k}K_{\beta_k}\mid 1\leq k\leq N\},
\end{align*}
the calligraphic letter being chosen so as
not to conflict with the braid group automorphisms $T_i$ of Theorem \ref{thm:Lusztigquantumrep}, and we write forbidden decompositions as pairs $(\bfa,\bfb)$ of exponent vectors, as after
Proposition \ref{prop:nobaddec}. By Remark \ref{p74}, $\mathcal{T}_{\bf i}$ depends only on the commutation class
of ${\bf i}$.

\begin{prop}\label{prop:onlyifdir}
Suppose that some ${\bf i}'$ in the commutation class of ${\bf i}$ admits, for some $k$,
a forbidden decomposition $(\bfa,\bfb)$ of $\beta_k$ such that the coefficient of
$e_{\beta_k}$ in the expansion of $e^{\bfb}e^{\bfa}$ in the PBW basis of
$U(\mathfrak{n}^+)$ is different from zero; here the $e_\beta\in\g_\beta$, $\beta\in\Dp$, are elements of a Chevalley basis of $\g$, and the roots $\beta_k$ and the monomials $e^{\bfa}=e_{\beta_1}^{a_1}\cdots e_{\beta_N}^{a_N}$ are defined by means of ${\bf i}'$. Then $\mathcal{T}_{\bf i}$ is not a quantum tangent
space.
\end{prop}
\begin{proof}
Write $E^{\bfb}E^{\bfa}=\sum_{\bfc}\lambda_{\bfc}E^{\bfc}$, the $\lambda_{\bfc}$ being
Laurent polynomials in $q$ by the integrality recalled after Theorem \ref{copro}. Specializing at $q=1$, under which the positive part of the integral form of
\cite[\S 1]{DCKP} becomes $U(\mathfrak{n}^+)$, the quantum root vectors go to the elements of a Chevalley basis up
to sign, so that $\lambda_{\bfe_k}(1)$ equals, up to sign, the classical coefficient,
which is nonzero by hypothesis. Hence $\lambda_{\bfe_k}\neq0$, that is $E_{\beta_k}$
occurs in $E^{\bfb}E^{\bfa}$, and $c_{\bfa,\bfb}\neq0$ by Corollary
\ref{cor:criterion}. Since $\sum_j a_j\geq2$, Proposition \ref{prop:qtscriterion} shows
that $\mathcal{T}_{{\bf i}'}$ is not a quantum tangent space, and neither is
$\mathcal{T}_{\bf i}=\mathcal{T}_{{\bf i}'}$ by Remark \ref{p74}.
\end{proof}

The characterization we are after now reduces to a single assertion in which no quantum
group occurs, and which we prove in Section \ref{sec:towards}.

\begin{thm}\label{conj:C2}
If ${\bf i}$ is not a tangent reduced decomposition, then some order of its commutation
class admits a forbidden decomposition whose classical coefficient, in the sense of
Proposition \ref{prop:onlyifdir}, is different from zero.
\end{thm}
Theorem \ref{conj:C2} is proved in Section \ref{sec:towards} (Theorem \ref{thm:C2ABCDE} and Proposition
\ref{prop:FG}). In every case the decomposition is found in the order induced by ${\bf i}$
itself, and is a special triple in the sense of Definition \ref{def:goodtriple} below.

We can now state the main result of the paper.
\begin{thm}\label{thm:combinatorialqts}
Let $w_0=s_{i_1}\cdots s_{i_N}$ be a reduced decomposition. Then the set
$\{F^{\bf i}_{\beta_k}K_{\beta_k}\mid 1\leq k\leq N\}$ spans a quantum
tangent space if and only if ${\bf i}$ is a tangent reduced decomposition.
\end{thm}
\begin{proof}
Assume that ${\bf i}$ is tangent. By Corollary \ref{conj:C1} its commutation class contains an
order satisfying (TR2$'$), so that $\mathcal{T}_{\bf i}$ is a quantum tangent space by Corollary
\ref{cor:tangentimpliesqts} and Remark \ref{p74}. If ${\bf i}$ is not tangent, Theorem
\ref{conj:C2} and Proposition \ref{prop:onlyifdir} show that $\mathcal{T}_{\bf i}$ is not a
quantum tangent space.
\end{proof}
In particular, in types $E_8$, $F_4$ and $G_2$ no reduced decomposition of $w_0$ gives rise to a quantum tangent
space (see the remark after Theorem \ref{t43}).

\section{\texorpdfstring{Proof of Theorem \ref{conj:C2}}{Proof of the reduction theorem}}\label{sec:towards}
Half of the statement is
available in general: if ${\bf i}$ is not tangent then, by Theorem
\ref{thm:tr2ptangent}, the order it induces does not satisfy (TR2$'$), so that by
Proposition \ref{prop:nobaddec} some $\beta_k$ admits a forbidden decomposition. What is
missing is that one of these decompositions has a nonzero classical coefficient, in the
sense of Proposition \ref{prop:onlyifdir}. The coefficient is easily computed for the
forbidden decompositions of the simplest possible shape, those with two summands
preceding $\beta_k$ and one following it, provided a mild extra condition holds. We give
these a name.

\begin{defn}\label{def:goodtriple}
Let $<$ be a convex order on $\Dp$ and let $\beta\in\Dp$. A \emph{special triple at $\beta$}
is a triple $(\gamma_1,\gamma_2,\psi)$ of positive roots, not necessarily distinct, such
that
\begin{align*}
\beta=\gamma_1+\gamma_2+\psi,\qquad \gamma_1\leq\gamma_2<\beta<\psi,\qquad
\beta-\gamma_2\in\Dp .
\end{align*}
We say that $<$ \emph{admits a special triple} if there is a special triple at some
$\beta\in\Dp$.
\end{defn}
A special triple at $\beta=\beta_k$ is in particular a forbidden decomposition of $\beta_k$,
the two summands $\gamma_1,\gamma_2$ preceding $\beta_k$ and $\psi$ following it; the
extra condition is that $\beta_k-\gamma_2=\gamma_1+\psi$ be a root. The next proposition
computes the classical coefficient of a special triple; the rest of the section produces
special triples.

\begin{prop}\label{prop:nonvanishing}
Let $(\gamma_1,\gamma_2,\psi)$ be a special triple at $\beta_k$. Then the
coefficient of $e_{\beta_k}$ in the expansion of $e_{\psi}e_{\gamma_1}e_{\gamma_2}$ in
the PBW basis of $U(\mathfrak{n}^+)$ equals
$N_{\psi,\gamma_1}N_{\beta_k-\gamma_2,\,\gamma_2}\neq0$, where the structure constants $N_{\a,\b}$ are defined by $[e_\a,e_\b]=N_{\a,\b}e_{\a+\b}$, with $N_{\a,\b}=0$ if $\a+\b\notin\D$. Consequently
$\mathcal{T}_{\bf i}$ is not a quantum tangent space.
\end{prop}
\begin{proof}
Since $[e_\psi,e_{\gamma_i}]=N_{\psi,\gamma_i}e_{\psi+\gamma_i}$, we have
\begin{align*}
e_{\psi}e_{\gamma_1}e_{\gamma_2}=e_{\gamma_1}e_{\gamma_2}e_{\psi}
+N_{\psi,\gamma_2}\,e_{\gamma_1}e_{\psi+\gamma_2}
+N_{\psi,\gamma_1}\,e_{\psi+\gamma_1}e_{\gamma_2}.
\end{align*}
If $\psi+\gamma_2$ is not a root the second summand vanishes; otherwise, by convexity $\gamma_2<\psi+\gamma_2$, so $\gamma_1<\psi+\gamma_2$ and the second
summand is already a PBW monomial. Again by convexity $\beta_k$ lies between
$\psi+\gamma_1=\beta_k-\gamma_2$ and $\gamma_2$, and $\gamma_2<\beta_k$, whence
$\gamma_2<\psi+\gamma_1$ and
$e_{\psi+\gamma_1}e_{\gamma_2}=e_{\gamma_2}e_{\psi+\gamma_1}
+N_{\psi+\gamma_1,\gamma_2}e_{\beta_k}$. The stated coefficient follows, and it is
nonzero because $\psi+\gamma_1=\beta_k-\gamma_2$ and $(\psi+\gamma_1)+\gamma_2=\beta_k$
are roots.
The argument covers the case $\gamma_1=\gamma_2$ verbatim. The last assertion is Proposition \ref{prop:onlyifdir}.
\end{proof}
\begin{prop}\label{prop:FG}
Let $\mathfrak g$ be of type $G_2$, with $\a_1$ short. Then every convex order admits a special triple, and
Theorem \ref{conj:C2} holds.
\end{prop}
\begin{proof}
The two convex orders are $\a_1<3\a_1+\a_2<2\a_1+\a_2<3\a_1+2\a_2<\a_1+\a_2<\a_2$ and its reverse. In the
first, $(\a_1,\a_1,\a_2)$ is a special triple at $2\a_1+\a_2$, as $(2\a_1+\a_2)-\a_1=\a_1+\a_2$ is a root; in the
second, $(\a_1+\a_2,\a_1+\a_2,\a_1)$ is a special triple at $3\a_1+2\a_2$, as
$(3\a_1+2\a_2)-(\a_1+\a_2)=2\a_1+\a_2$ is a root. Proposition \ref{prop:nonvanishing} concludes.
\end{proof}
For the other types special triples are produced by the following construction. We shall use an elementary
fact.
\begin{lem}\label{lem:splitroot}
Let $\gamma,\sigma\in\Dp$ be such that $\gamma+\sigma\in\Dp$, and let $\sigma=\mu+\nu$ with
$\mu,\nu\in\Dp$. Then $\gamma+\mu$ or $\gamma+\nu$ is a root.
\end{lem}
\begin{proof}
The Jacobi identity $[e_\gamma,[e_\mu,e_\nu]]=[[e_\gamma,e_\mu],e_\nu]+[e_\mu,[e_\gamma,e_\nu]]$ shows that if $\gamma,\mu,\nu$ are arbitrary roots
such that $\mu+\nu$ and $\gamma+\mu+\nu$ are roots, then $\gamma+\mu$ or $\gamma+\nu$ lies in
$\D\cup\{0\}$, since the left-hand side is nonzero. For positive roots the value $0$ cannot occur.
\end{proof}

Special triples may be produced canonically. Given a non-simple $\beta\in\Dp$, put
\begin{align*}
M(\beta)=\{\mu\in\Dp\mid \mu<\beta,\ \beta-\mu\in\Dp\},
\end{align*}
a nonempty set: writing $\beta$ as a sum of two positive roots, convexity places $\beta$
between them, so that one of them precedes it. By Lemma \ref{lem:proj}(a),
$M(\beta)\subseteq \Phi(p_<(\beta))$.

\begin{prop}\label{prop:maxsplit}
Let $\beta\in\Dp$ be non-simple, let $\gamma_2$ be the largest element of $M(\beta)$ with
respect to $<$, and put $\delta=\beta-\gamma_2$. If $\delta=\gamma_1+\psi$ with $\gamma_1,\psi\in\Dp$ and
$\gamma_1<\psi$, then $\gamma_1\leq\gamma_2$, and $(\gamma_1,\gamma_2,\psi)$ is a special triple at $\beta$.
In particular, if $\delta$ is not simple, then $\mathcal{T}_{\bf i}$ is not a quantum tangent space.
\end{prop}
\begin{proof}
Note that $\beta<\delta$ by convexity, as $\gamma_2<\beta$. Suppose $\gamma_2<\gamma_1$. By Lemma
\ref{lem:splitroot}, applied to $\delta=\gamma_1+\psi$, one of $\gamma_1+\gamma_2$ and $\gamma_2+\psi$ is a root.
\par If $\zeta:=\gamma_1+\gamma_2$ is a root, then $\beta=\zeta+\psi$, and convexity places $\zeta$ strictly
between $\gamma_2$ and $\gamma_1$, so $\gamma_2<\zeta<\gamma_1<\psi$. If $\zeta<\beta$, then
$\zeta\in M(\beta)$ and $\zeta>\gamma_2$, against the choice of $\gamma_2$; if $\zeta>\beta$, then $\psi<\beta$,
against $\psi>\zeta>\beta$.
\par If $\xi:=\gamma_2+\psi$ is a root, then $\beta=\xi+\gamma_1$. If $\gamma_1<\beta$, then
$\gamma_1\in M(\beta)$ and $\gamma_1>\gamma_2$, against the choice of $\gamma_2$. If $\gamma_1>\beta$, then
$\xi<\beta$, so that $\xi\in M(\beta)$ and $\xi<\gamma_2$ by maximality; convexity places $\xi$ strictly
between $\gamma_2$ and $\psi$, whence $\psi<\xi<\gamma_2<\beta<\gamma_1$, against $\gamma_1<\psi$.
\par Hence $\gamma_1\leq\gamma_2<\beta$. Convexity places $\delta$ strictly between $\gamma_1$ and $\psi$, so
$\gamma_1<\delta<\psi$ and $\psi>\delta>\beta$; finally $\beta-\gamma_2=\delta$ is a root. Thus
$(\gamma_1,\gamma_2,\psi)$ is a special triple at $\beta$, and the last assertion follows from Proposition
\ref{prop:nonvanishing}, since a non-simple $\delta$ is the sum of two positive roots, which are distinct,
twice a root being never a root.
\end{proof}

\begin{remark}\label{rem:tr2triplecheck}
Let $\beta_{k_0}$ be the \emph{earliest failing root} of a non-tangent convex order, the first root at
which (TR1) or (TR2) fails. It is not simple: at a simple root $\beta_k=\a_i$ one has $p_<(\beta_k)=s_i$, and
both conditions hold trivially. By Proposition \ref{prop:maxsplit}, Theorem \ref{conj:C2} holds for the
order under consideration as soon as $\beta_{k_0}-\max M(\beta_{k_0})$ is not simple; that this is always the
case outside type $G_2$ is the content of Theorem \ref{thm:starstar}.
\end{remark}

\subsection{\texorpdfstring{Types other than $G_2$}{Types other than G2}}\label{sec:ABCDE}
Throughout this subsection $\D$ is not of type $G_2$. The only property of these types that we
use is that root strings have length at most three, which fails only in type $G_2$; consequently, for roots $\mu\neq\pm\nu$
with $\mu+\nu\in\D$,
\begin{itemize}
\item[(S)] $\langle\mu,\nu^\vee\rangle\leq0$,
\end{itemize}
since the $\nu$-string $\mu-p\nu,\dots,\mu+q\nu$ through $\mu$ satisfies
$\langle\mu,\nu^\vee\rangle=p-q$, contains $\mu+\nu$ and has length at most three, so that $q\geq1$
and $p+q\leq2$.
\par We fix $k$ such that (TR1) and (TR2) hold at $\beta_1,\dots,\beta_{k-1}$, the root
$\beta_k$ is not simple, and $\delta=\beta_k-\gamma_2$ is simple, where
$\gamma_2=\max M(\beta_k)=\beta_i$ with $i<k$. We put
\begin{equation}\label{defP}
P_2=\Phi(p_<(\gamma_2)),\qquad P=\Phi(p_<(\beta_k)),
\end{equation}
the projections at $\gamma_2$ and at $\beta_k$, and we denote by $c_\delta$ the coefficient of the simple root
$\delta$. Note that $\beta_k<\delta$ by convexity, and that $\delta\notin P_2$: indeed $\delta$ follows $\gamma_2=\beta_i$, hence
it does not lie in $\Phi(w_i)\supseteq P_2$.
\par The strategy is the following. Both {\rm (TR1)} and {\rm (TR2)} at $\beta_k$ are assertions about the
single set $P$, the smallest biconvex subset of $\Phi(w_k)$ containing $\beta_k$ (Proposition \ref{p32}). The
inductive hypothesis describes $P_2$ completely: it consists of the positive roots $\a\leq_0\gamma_2$ with
coefficient one on a certain simple root $\a_t$ (Lemma \ref{lem:Qprops}). The natural candidate for $P$ is the
analogous set at $\beta_k$,
\begin{equation}\label{defX}
X=\{\a\in\Dp\mid \a\leq_0\beta_k,\ c_t(\a)=1\},
\end{equation}
which contains $P_2$, its other elements being of the form $\delta+x$ with $x\in P_2$ (Lemma \ref{lem:Xdelta}). We
show that $X$ is an abelian biconvex subset of $\Phi(w_k)$ containing $\beta_k$ (Proposition \ref{prop:Xbiconvex});
minimality then gives $P\subseteq X$, from which {\rm (TR1)}, {\rm (TR2)} at $\beta_k$ and $P=X$ follow
(Theorem \ref{thm:starstar}). The type of $\D$ enters only through (S), in Lemma \ref{lem:xgamma}, which places
$X$ inside $\Phi(w_k)$.

\par\noindent{\bf Example.} Let $\D$ be of type $D_4$, with $\a_2$ the central
simple root, and let $<$ be the convex order induced by
$w_0=s_1s_2s_3s_4s_2s_1s_2s_3s_2s_4s_2s_3$, namely
\begin{align*}
&\a_1<\a_1{+}\a_2<\a_1{+}\a_2{+}\a_3<\a_1{+}\a_2{+}\a_4<\a_1{+}\a_2{+}\a_3{+}\a_4<\theta
<\\&\a_2<\a_2{+}\a_4<\a_4<\a_2{+}\a_3{+}\a_4<\a_2{+}\a_3<\a_3 ,
\end{align*}
where $\theta=\a_1+2\a_2+\a_3+\a_4$. Conditions {\rm (TR1)} and {\rm (TR2)} hold at
$\beta_1,\dots,\beta_5$, and at $\beta_6=\theta$ one has
$M(\theta)=\{\a_1{+}\a_2,\ \a_1{+}\a_2{+}\a_3,\ \a_1{+}\a_2{+}\a_4,\ \gamma_2\}$ with
$\gamma_2=\a_1{+}\a_2{+}\a_3{+}\a_4$, so that $\delta=\theta-\gamma_2=\a_2$ is simple. Here
\[
P_2=\{\a_1,\ \a_1{+}\a_2,\ \a_1{+}\a_2{+}\a_3,\ \a_1{+}\a_2{+}\a_4,\ \gamma_2\},
\]
the simple root $\a_t$ is $\a_1$, and $X=P_2\cup\{\theta\}$ is the set of all the positive roots with
coefficient one on $\a_1$; one checks that $P=X$.

\begin{lem}\label{lem:principal}
Suppose that {\rm (TR1)} and {\rm (TR2)} hold at $\beta_1,\dots,\beta_j$. Then
$\Phi(p_<(\beta_j))$ contains exactly one simple root $\a_t$, and
\begin{align*}
\Phi(p_<(\beta_j))=\{\a\in\Dp\mid \a\leq_0\beta_j,\ c_t(\a)=1\},
\end{align*}
where $c_t$ denotes the coefficient of $\a_t$.
\end{lem}
\begin{proof}
Put $\Phi_j=\Phi(p_<(\beta_j))$. By Proposition \ref{prop:sumfree} one has $\Phi_j\subseteq\beta_j-Q^+$,
and by Lemma \ref{lem:simpleprec}(b) $\Phi_j$ contains a simple root $\a_t$. For $\a\in \Phi_j$ with
$\a\neq\beta_j$ one has $\a<_0\beta_j$, and (TR2) at $\beta_j$ gives
$\a_t\notin\supp(\beta_j-\a)$, that is $c_t(\a)=c_t(\beta_j)$. If $\beta_j\neq\a_t$, the
choice $\a=\a_t$ gives $c_t(\beta_j)=1$, which is obvious if $\beta_j=\a_t$. Hence every
element of $\Phi_j$ has coefficient one on $\a_t$, and in particular $\a_t$ is the only simple root
in $\Phi_j$.
\par Conversely, let $\a\in\Dp$ with $\a\leq_0\beta_j$ and $c_t(\a)=1$. Recall that two
positive roots $\a\leq_0\beta$ can be joined by a chain
$\a=\a^0\leq_0\a^1\leq_0\dots\leq_0\a^m=\beta$ of positive roots in which each difference
$\a^h-\a^{h-1}$ is a simple root\footnote{There is nothing to prove if $\a=\beta$. Otherwise write $\beta-\a=\sum_jc_j\a_j$; as
$(\beta-\a,\beta-\a)>0$, some $j$ with $c_j>0$ satisfies $(\beta-\a,\a_j)>0$, so that either
$(\a,\a_j)<0$ and $\a+\a_j$ is a root, or $(\beta,\a_j)>0$ and $\beta-\a_j$ is a root. In
either case one obtains a positive root lying between $\a$ and $\beta$ and differing from one
of them by $\a_j$ (note that $\beta\neq\a_j$, as $\beta-\a_j-\a\in Q^+$), and one concludes by
induction on $\height(\beta)-\height(\a)$.}. For the
chain from $\a$ to $\beta_j$ all the terms have coefficient one on $\a_t$, so that the simple
roots $\a^h-\a^{h-1}$ are different from $\a_t$ and do not lie in $\Phi_j$. As $\Phi_j$ is co-closed,
$\a^h\in \Phi_j$ implies $\a^{h-1}\in \Phi_j$, and descending induction from $\a^m=\beta_j$ gives
$\a\in \Phi_j$.
\end{proof}

\begin{lem}\label{lem:Qprops}
\begin{enumerate}
\item[(a).] The set $P_2$ is abelian and contained in $\gamma_2-Q^+$.
\item[(b).] $\Phi(p_<(y))\subseteq P_2$ for every $y\in P_2$.
\item[(c).] The set $P_2$ contains exactly one simple root $\a_t$, and
$P_2=\{\a\in\Dp\mid \a\leq_0\gamma_2,\ c_t(\a)=1\}$.
\end{enumerate}
\end{lem}
\begin{proof}
Conditions (TR1) and (TR2) hold at $\beta_1,\dots,\beta_i$, so that (a) and (b) are Proposition
\ref{prop:sumfree} and Lemma \ref{lem:proj}(b) at $\gamma_2$, and (c) is Lemma
\ref{lem:principal} at $\gamma_2$.
\end{proof}
From now on $\a_t$ denotes the simple root of Lemma \ref{lem:Qprops}(c), and $X$ the set defined above. As
$\delta\notin P_2$, we have $\delta\neq\a_t$, so that $c_t(\delta)=0$ and $c_t(\beta_k)=c_t(\gamma_2)=1$; in particular
$\beta_k\in X$ and $\delta\notin X$. Since $\beta_k=\gamma_2+\delta$, an element of $X$ lies below $\gamma_2$ if and
only if its coefficient on $\delta$ is smaller than $c_\delta(\beta_k)$, so that, by Lemma \ref{lem:Qprops}(c),
\[
P_2=\{\a\in X\mid c_\delta(\a)<c_\delta(\beta_k)\},\qquad X\setminus P_2=\{\a\in X\mid c_\delta(\a)=c_\delta(\beta_k)\}.
\]

\begin{lem}\label{lem:Xdelta}
For every $y\in X\setminus P_2$, the difference $y-\delta$ is a positive root lying in $P_2$.
\end{lem}
\begin{proof}
We argue by descending induction on $\height(y)$. If $y=\beta_k$, then $y-\delta=\gamma_2$ and there is nothing
to prove. Otherwise there is a simple root $\a_j$ such that $y+\a_j$ is a positive root and
$y+\a_j\leq_0\beta_k$, by the fact recalled in the proof of Lemma \ref{lem:principal}. The coefficients of
$y+\a_j$ on $\a_t$ and on $\delta$ lie between those of $y$ and those of $\beta_k$, which are equal; hence
$\a_j\neq\a_t,\delta$ and $y+\a_j\in X\setminus P_2$, so that $x':=y+\a_j-\delta$ is a root by induction. The general
form of Lemma \ref{lem:splitroot}, applied to $-\a_j$ and to the root $y+\a_j=\delta+x'$, shows that $\delta-\a_j$
or $x'-\a_j$ lies in $\D\cup\{0\}$. The former is impossible, $\delta$ and $\a_j$ being distinct simple roots, and
$x'-\a_j=y-\delta$ is not zero, as $c_t(y-\delta)=1$. Hence $y-\delta$ is a root, positive because $c_t(y-\delta)=1$.
Finally $y\leq_0\beta_k=\gamma_2+\delta$ gives $y-\delta\leq_0\gamma_2$, so that
$y-\delta\in P_2$ by Lemma \ref{lem:Qprops}(c).
\end{proof}

\begin{lem}\label{lem:xgamma}
Let $y\in X\setminus(P_2\cup\{\beta_k\})$ and $x=y-\delta$. Then $(x,\gamma_2)>0$; consequently
$\gamma_2-x\in\Dp$ and $y<\beta_k$. In particular $X\subseteq \Phi(w_k)$.
\end{lem}
\begin{proof}
By Lemma \ref{lem:Xdelta} $x\in P_2$, and $x\neq\gamma_2$ as $y\neq\beta_k$; moreover $x+\delta=y$ is a root.
As $x+\gamma_2$ is not a root, $P_2$ being abelian, $(x,\gamma_2)\geq0$; and $\beta_k+x$
is not a root by Lemma \ref{lem:nosum}, since $x<\beta_k$ (as $x\in P_2\subseteq \Phi(w_i)$) and
$\supp(x)\subseteq\supp(\gamma_2)\subseteq\supp(\beta_k)$ by Lemma \ref{lem:Qprops}(a),
whence also $(x,\beta_k)\geq0$.
\par Suppose $(x,\gamma_2)=0$. Then $s_{\gamma_2}$ fixes $x$ and interchanges $x+\gamma_2$
with $x-\gamma_2$, so that $x-\gamma_2$ is not a root either. Moreover
$(x,\delta)=(x,\beta_k)-(x,\gamma_2)\geq0$, while $\langle x,\delta^\vee\rangle\leq0$ by (S),
the sum $x+\delta$ being a root and $x\neq\delta$ because $\delta\notin P_2$; hence $(x,\delta)=(x,\beta_k)=0$.
Since $x$ is orthogonal to $\delta$, $\gamma_2$ and $\beta_k$,
\begin{align*}
(y,\gamma_2)=(\delta,\gamma_2),\qquad
(y,\beta_k)=(\delta,\beta_k)=(\delta,\delta)+(\delta,\gamma_2),
\end{align*}
and $(\delta,\gamma_2)\leq0$ by (S), $\delta+\gamma_2=\beta_k$ being a root and $\delta\neq\gamma_2$. If
$(\delta,\gamma_2)<0$, then $(y,\gamma_2)<0$ and $y+\gamma_2=\beta_k+x$ is a root,
which is excluded. If $(\delta,\gamma_2)=0$, then $(y,\beta_k)=(\delta,\delta)>0$ and
$y\neq\beta_k$, so that $y-\beta_k=x-\gamma_2$ is a root, which
is excluded as well. Hence $(x,\gamma_2)>0$, and $\gamma_2-x$ is a
root, positive because $x<_0\gamma_2$.
\par The root $\gamma_2$ is the largest element of $P_2$, as
$P_2\subseteq \Phi(w_i)$ and $\gamma_2=\beta_i$, so $x<\gamma_2$; convexity applied to
$\gamma_2=x+(\gamma_2-x)$ places $\gamma_2$ strictly between the two summands, whence
$\gamma_2<\gamma_2-x$. Now
\begin{align*}
\beta_k=(\delta+x)+(\gamma_2-x)
\end{align*}
is a decomposition of $\beta_k$ into two positive roots, so exactly one of them precedes
$\beta_k$. Were it $\gamma_2-x$, we would have $\gamma_2-x\in M(\beta_k)$, because
$\beta_k-(\gamma_2-x)=\delta+x\in\Dp$, and the maximality of $\gamma_2$ would give
$\gamma_2-x\leq\gamma_2$, against $\gamma_2<\gamma_2-x$. Hence $y=\delta+x<\beta_k$.
\par Finally $P_2\subseteq \Phi(w_i)\subseteq \Phi(w_k)$, while every element of $X\setminus P_2$ is
either $\beta_k$ or precedes it by the above; so $X\subseteq \Phi(w_k)$.
\end{proof}

\begin{lem}\label{lem:alphat}
\begin{enumerate}
\item[(a).] Every $y\in X\setminus\{\beta_k\}$ precedes $\beta_k$, and $\a_t\in \Phi(p_<(y))$.
\item[(b).] Let $\eta\in\Dp$ with $\eta<\beta_k$, and let $\eta=y+z$ with $y,z\in\Dp$ and
$y<\eta<z$. Then $\Phi(p_<(y))\subseteq \Phi(p_<(\eta))$, and if $\a_t\in \Phi(p_<(\eta))$, then $c_t(z)=0$.
\end{enumerate}
\end{lem}
\begin{proof}
(a). Let $y\in X\setminus\{\beta_k\}$. If $y\in P_2$, then $y\leq\gamma_2<\beta_k$, and $\Phi(p_<(y))\subseteq P_2$ by Lemma
\ref{lem:Qprops}(b); by Lemma \ref{lem:simpleprec}(b) the set $\Phi(p_<(y))$ contains a simple
root, which lies in $P_2$ and therefore equals $\a_t$. Otherwise $y\in X\setminus(P_2\cup\{\beta_k\})$, so that
$x:=y-\delta\in P_2$ by Lemma \ref{lem:Xdelta}, and $y<\beta_k$ by Lemma \ref{lem:xgamma}, so that {\rm (TR1)}
holds at $y$. Convexity applied to $y=x+\delta$ places $y$ strictly between $x$ and $\delta$, and
$y<\beta_k<\delta$, so that $x<y<\delta$. Lemma \ref{lem:proj}(a) gives $x\in \Phi(p_<(y))$, and then
Lemma \ref{lem:proj}(b) gives $\Phi(p_<(x))\subseteq \Phi(p_<(y))$; as in  the first case,
$\a_t\in \Phi(p_<(x))$.
\par (b). By Lemma \ref{lem:proj}(a) $y\in \Phi(p_<(\eta))$, so that
$\Phi(p_<(y))\subseteq \Phi(p_<(\eta))$ by Lemma \ref{lem:proj}(b). Moreover $y<_0\eta$ with
$\eta-y=z$; so {\rm (TR2)} at $\eta$ gives $\supp(z)\cap \Phi(p_<(\eta))=\emptyset$, that is $\a_t\notin\supp(z)$.
\end{proof}

\begin{prop}\label{prop:Xbiconvex} Let $P, X$ be as in \eqref{defP}, \eqref{defX}, respectively.
The set $X$ is an abelian biconvex subset of $\Phi(w_k)$ containing $\beta_k$. Consequently $P\subseteq X$; in
particular $P$ is abelian.
\end{prop}
\begin{proof}
We have seen that $\beta_k\in X$, and $X\subseteq \Phi(w_k)$ by Lemma \ref{lem:xgamma}. Suppose
that $y_1,y_2\in X$ and that $\eta=y_1+y_2$ is a root; then $y_1\neq y_2$, and convexity places $\eta$ strictly
between them; label them so that $y_1<\eta<y_2$. Then $\eta<y_2\leq\beta_k$ and $y_1\neq\beta_k$, so that $\a_t\in \Phi(p_<(y_1))\subseteq \Phi(p_<(\eta))$ by
Lemma \ref{lem:alphat}, whose part (b) then gives $c_t(y_2)=0$, against $y_2\in X$. Hence $X$ is abelian, and a fortiori closed. It is co-closed: if $y\in X$ and $y=y'+y''$ with
$y',y''\in\Dp$, then $c_t(y')+c_t(y'')=1$, so that one of the two summands, say $y'$, has coefficient one on
$\a_t$; as $y'\leq_0y\leq_0\beta_k$, we get $y'\in X$. Thus $X$ is biconvex. As $P$
is the smallest biconvex subset of $\Phi(w_k)$ containing $\beta_k$, we get $P\subseteq X$.
\end{proof}

\begin{thm}\label{thm:starstar}
Let $\beta_k$ be a non-simple root. Assume that {\rm (TR1)} and {\rm (TR2)} hold at all the roots preceding
$\beta_k$, and that $\beta_k-\max M(\beta_k)$ is simple. Then {\rm (TR1)} and {\rm (TR2)} hold at $\beta_k$, and
$\Phi(p_<(\beta_k))=X$.
\end{thm}
\begin{proof}
By Proposition \ref{prop:Xbiconvex}, $P\subseteq X$; and the only simple root lying in $X$, hence possibly in
$P$, is $\a_t$.
\par {\rm (TR2)}. Let $\nu\in P$ with $\beta_k-\nu\in Q^+\setminus\{0\}$. As $\nu\in X$, one has
$c_t(\beta_k-\nu)=0$, so that $\supp(\beta_k-\nu)\cap P=\emptyset$.
\par {\rm (TR1)}. Otherwise, by Lemma \ref{lem:tr1reform}, there are $\mu,\nu\in\Dp$ with
$\nu<\mu+\nu<\mu\leq\beta_k$ and $\mu,\mu+\nu\in P\subseteq X$. As $\mu+\nu\in X\setminus\{\beta_k\}$, Lemma
\ref{lem:alphat}(a) gives $\a_t\in \Phi(p_<(\mu+\nu))$, and Lemma \ref{lem:alphat}(b), applied to
$\eta=\mu+\nu=\nu+\mu$, gives $c_t(\mu)=0$, against $\mu\in X$.
\par Finally, the conditions (TR1), (TR2) hold at $\beta_1,\dots,\beta_k$, so that Lemma \ref{lem:principal} describes $P$ by
means of its only simple root; as $P\subseteq X$, that root is $\a_t$, and then $P=X$ by the definition of $X$.
\end{proof}

We may now draw the following conclusion.
\begin{thm}\label{thm:C2ABCDE}
Theorem \ref{conj:C2} holds when $\mathfrak g$ is not of type $G_2$. More
precisely, let $<$ be a non-tangent convex order, let $\beta_{k_0}$ be its earliest
failing root and put $\gamma_2=\max M(\beta_{k_0})$. Then there are $\gamma_1,\psi\in\Dp$
such that $(\gamma_1,\gamma_2,\psi)$ is a special triple at $\beta_{k_0}$, so that
$\mathcal T_{\bf i}$ is not a quantum tangent space.
\end{thm}
\begin{proof}
The root $\beta_{k_0}$ is not simple by Remark \ref{rem:tr2triplecheck}, and $\delta=\beta_{k_0}-\gamma_2$ is not
simple by Theorem \ref{thm:starstar}. Writing $\delta=\gamma_1+\psi$ with $\gamma_1<\psi$, Proposition
\ref{prop:maxsplit} concludes.
\end{proof}

\vskip5pt
\noindent{\bf AI-generated content disclosure.}

Claude has been used to double-check the literature on our results of Section \ref{sec:coproduct}, and to refine condition (TR2) into condition (TR2$'$).\par Of course, we take full responsibility for the final content of this work.

\vskip5pt
{\bf Acknowledgments.} This work was partially supported by the NSERC grant ``Combinatorics of infinite Coxeter
groups'' held by C.~Hohlweg. \par
A.C. acknowledges support by the GA\v{C}R EXPO project ``Homotopick\'e a homologick\'e metody a n\'astroje \'uzce souvisej\'{\i}c\'{\i} s matematickou fyzikou'', by MSCA doctoral network 2022 CaLiForNIA, by an INdAM Research grant, and finally by ``Implementation of cutting-edge research and its application as part of the Scientific Center of Excellence for Quantum and Complex Systems, and Representations of Lie Algebras'', Grant No. PK.1.1.10.0004, co-financed by the European Union through the European Regional Development Fund -- Competitiveness and Cohesion Programme 2021--2027. A.C. is grateful for the hospitality offered by Sapienza University of Rome. We wish to thank Giovanna Carnovale, Fabio Gavarini and R\'eamonn \'O Buachalla for useful discussions.

\end{document}